\documentclass{amsart}
\usepackage{amssymb, amstext, amscd, amsmath,color}
\makeatletter
\def\@cite#1#2{{\m@th\upshape\bfseries%
[{#1\if@tempswa{\m@th\upshape\mdseries, #2}\fi}]}}
\makeatother

\theoremstyle{plain}
\newtheorem{thm}{Theorem}[section]
\newtheorem{cor}[thm]{Corollary}

\newtheorem{lem}[thm]{Lemma}

\theoremstyle{definition}
\newtheorem{rem}[thm]{Remark}

\newtheorem{defn}[thm]{Definition}

\newtheorem{eg}[thm]{Example}

\newcommand{\bA}{{\mathbb{A}}}
\newcommand{\bB}{{\mathbb{B}}}
\newcommand{\bC}{{\mathbb{C}}}
\newcommand{\bD}{{\mathbb{D}}}

  \newcommand{\B}{{\mathcal{B}}}

\renewcommand{\H}{{\mathcal{H}}}

  \newcommand{\M}{{\mathcal{M}}}

  \newcommand{\U}{{\mathcal{U}}}

\newcommand{\ep}{\varepsilon}
\renewcommand{\phi}{\varphi}
\newcommand{\upchi}{{\raise.35ex\hbox{\ensuremath{\chi}}}}

\newcommand{\fA}{{\mathfrak{A}}}
\newcommand{\fB}{{\mathfrak{B}}}
\newcommand{\fC}{{\mathfrak{C}}}
\newcommand{\fD}{{\mathfrak{D}}}

\newcommand{\fJ}{{\mathfrak{J}}}

\newcommand{\fL}{{\mathfrak{L}}}
\newcommand{\fM}{{\mathfrak{M}}}

\newcommand{\fR}{{\mathfrak{R}}}

\newcommand{\qand}{\quad\text{and}\quad}
\newcommand{\qif}{\quad\text{if}\quad}
\newcommand{\qfor}{\quad\text{for}\quad}
\newcommand{\qforal}{\quad\text{for all}\quad}

\newcommand{\AND}{\text{ and }}

\newcommand{\FORAL}{\text{ for all }}

\newcommand{\Alg}{\operatorname{Alg}}

\newcommand{\dist}{\operatorname{dist}}

\newcommand{\Lat}{\operatorname{Lat}}
\newcommand{\CycLat}{\operatorname{CycLat}}
\newcommand{\ran}{\operatorname{Ran}}

\newcommand{\spn}{\operatorname{span}}

\newcommand{\bsl}{\setminus}

\newcommand{\ip}[1]{\langle #1 \rangle}

\newcommand{\ol}{\overline}
\newcommand{\one}{{\mathbf{1}}}

\newcommand{\wot}{\textsc{wot}}

\newcommand{\Fd}{\mathbb{F}_d^+}
\newcommand{\Fock}{{\ell^2(\Fd)}}
\newcommand{\Hinf}{H^\infty }

\begin{document}

%%%%%%%%%%%%%%%%%%%%%%%%%%%%%%%%%%%%%%
\title[NP Interpolation and Factorization]%
{Nevanlinna-Pick Interpolation and\\ Factorization of Linear Functionals}

\author[K.R. Davidson]{Kenneth R. Davidson}
\address{Pure Math.\ Dept.\\U. Waterloo\\Waterloo, ON\;
N2L--3G1\\CANADA}
\email{krdavids@uwaterloo.ca}

\author[R. Hamilton]{Ryan Hamilton}
\address{Pure Math.\ Dept.\\U. Waterloo\\Waterloo, ON\;
N2L--3G1\\CANADA}
\email{rhamilto@math.uwaterloo.ca}

\begin{abstract}
If $\fA$ is a unital weak-$*$ closed algebra of multiplication operators on a
reproducing kernel Hilbert space which has the property $\bA_1(1)$, then the
cyclic invariant subspaces index a Nevanlinna-Pick family of kernels.  This yields an
NP interpolation theorem for a wide class of algebras.  
In particular, it applies to many function spaces over the unit disk including
Bergman space.
We also show that the multiplier algebra of a complete NP space
has $\bA_1(1)$, and thus this result applies to all of its subalgebras.
A matrix version of this result is also established.  It applies, in particular,
to all unital weak-$*$ closed subalgebras of $H^\infty$ acting on Hardy space
or on Bergman space.
\end{abstract}


\thanks{\noindent First author was partially supported by an NSERC grant.}
\thanks{\noindent Second author was partially supported by an NSERC fellowship.}

\date{}
\maketitle

%%%%%%%%%%%%%%%%%%%%%%%%%%%%%%%%%%%%%%%%%%%
\section{Introduction}\label{S:intro}

The classical interpolation theorem for analytic functions is due to Pick in 1916.
Suppose $z_1, \dots, z_n$ are distinct points in the complex open disk $\bD$ 
and  $w_1, \dots, w_n$ are complex numbers. 
The Nevan\-linna-Pick (NP) interpolation theorem asserts that the positivity of the matrix
\begin{align*}
 \left [ \frac{1-w_i \ol{w_j}}{1-z_i \ol{z_j}} \right ]
\end{align*}
is a necessary and sufficient condition to ensure the existence of an analytic 
function $f$ on $\bD$ satisfying $f(z_i) = w_i$ for $i=1, \dots ,n$ and 
$\|f\|_\infty :=\sup \{ | f(z) |: z \in \bD \} \leq 1$.
Nevanlinna gave a new proof and provided a parameterization of all solutions
a few years later.

The seminal work of Sarason \cite{Sar67} reformulated the Nevan\-linna-Pick theorem 
in operator theoretic language.  
Let $\fJ_E$ be the ideal of functions in $H^\infty (\bD)$ that vanish on 
$E= \{z_1,\dots,z_n\}$, and let $M(E) = H^2 \ominus \ol{\fJ_E H^2}$.
Sarason showed that the NP theorem is equivalent to the representation
\begin{align*}
 f + \fJ_E \mapsto P_{M(E)} M_f P_{M(E)}
\end{align*}
being isometric, where $M_f$ denotes the multiplication operator 
by $f$ on Hardy space $H^2(\bD)$.  
He established this using a prototypical version of the commutant lifting theorem.
The Szeg\"{o} kernel $k^S_z(w) = (1-\ol{z}w)^{-1}$ is the 
reproducing kernel for $H^2(\bD)$, and one easily sees that 
$M(E) = \spn\{k^S_{z_1}, \dots, k^S_{z_n}  \}$.   
The operator $P_{M(E)} M_f^* P_{M(E)} = M_f^*|_{M(E)}$ \vspace{.2ex}
is diagonal with respect 
to the (non-orthogonal) basis $k_{\lambda_1}^S, \dots, k_{\lambda_n}^S$ 
since $M_f^* k^S_{z_i} = \ol{f(z_i)} k^S_{z_i}$.   
We can summarize the classic NP theorem as the distance formula
\begin{align*}
 \dist(f, \fJ_E) = \| M_f^*|_{M(E)} \| .
\end{align*}

In \cite{Abram79}, Abrahamse proved a Nevan\-linna-Pick type theorem 
on a multiply connected domain $A$.  
Here, a Pick matrix associated to a single kernel was not sufficient to guarantee the existence of a solution.  
Instead, an entire family of kernels indexed by copies of the complex torus was required.
These spaces arose as subspaces of $L^2(\partial A)$ which are rationally invariant.
In the case of the annulus, these subspaces were classified by Sarason \cite{Sar65}.
Analogous invariant subspaces exist for all nice finitely connected domains.
Abrahamse establishes a factorization theorem which shows that certain
linear functionals can be represented as rank one functionals on one of these subspaces.
This was the first appearance of a Nevan\-linna-Pick family of kernels 
in the literature, and motivated the search for other Nevan\-linna-Pick families 
associated to different algebras of functions. 

There is  considerable literature concerned with interpolation in the multiplier algebra 
$\fM(H)$ of a reproducing kernel Hilbert space $H$ over a set $X$.
Again, given $E=\{ \lambda_1,\dots,\lambda_n\}$ in $X$ and scalars $w_1,\dots,w_n$,
we are interested in finding a function $f\in \fM(H)$ such that $f(\lambda_i) = w_i$ for
$1 \le i \le n$ and $\| M_f \| \le 1$.
We again define $\fJ_E = \{ f \in \fM(H) : f|_E=0\}$ and 
$M(E) = \spn\{k_{\lambda_i} : \lambda_i \in E \}$.
Such a Hilbert space is called a Nevan\-linna-Pick kernel if the analogous 
distance formula holds:
\begin{align*}
 \dist(M_f, \fJ_E) = \| M_f^*|_{M(E)} \| .
\end{align*}
The kernel functions $k_{\lambda_i}$ form a basis for $M(E)$ consisting of eigenvectors
for $M_f^*$, and so the right hand side is at most $1$ if and only if the following matrix is positive
\[
 \big[
  (1- w_i\ol{w_j}) \ip{k_{\lambda_j},k_{\lambda_i}}
 \big] .
\]

One can also consider matrix interpolation, and the classical theorem works
just as well for matrix algebras over $H^\infty$.
A kernel for which the classical formula holds for all matrix valued functions
is called a complete NP kernel.
Results of McCullough\cite{McCull92,McCull94} and Quiggin \cite{Quig93,Quig94},
building on (unpublished) work by Agler, provide a classification of complete NP kernels.  
Davidson and Pitts \cite{DP98b} show that the Drury-Arveson space 
or symmetric Fock space on the unit ball $\bB_d$ of $\bC^d$ is a complete NP kernel.
Agler and McCarthy \cite{AMc00} showed that every complete NP kernel is equivalent
to the restriction of the Drury-Arveson space to some subset of the ball.
There is no known characterization of ordinary NP kernels.

In \cite{DPRS09}, a Nevan\-linna-Pick problem on the Hardy space
was studied for the subalgebra 
\[ H^\infty_1 = \{ f \in H^\infty(\bD) : f'(0) = 0\} .\]
Beurling's theorem for the shift was used to characterize the invariant 
subspace lattice for this algebra. 
The main result in \cite{DPRS09} furnishes a necessary and sufficient condition 
for the existence of an interpolant in $H^\infty_1$.  
This requires the positivity of an entire family of Pick matrices 
\begin{align*}
 \big[ (1-w_i \ol{w_j}) \ip{P_Lk^S_{z_j},k^S_{z_i}} \big]
\end{align*}
for (a family of) invariant subspaces $L$ of $H^\infty_1$.  
Again this is equivalent to a distance formula for the ideal
$\fJ_E = \{f \in \H^\infty_1 : f|_E=0\}$:
\begin{align*} %\label{E:distance}
 \dist(M_f, \fJ_E) =  \sup_{L \in \Lat(H^\infty_1)}
 \| P_{L \ominus \fJ_E L} M_f^*|_{L \ominus \fJ_E L}\| .
\end{align*}

In this paper, we are interested in general interpolation problems of this type.  
Consider a reproducing kernel Hilbert space $H$ with kernel $k$ on some set $X$,
and a unital weak-$*$-closed subalgebra $\fA$ of multipliers on $H$.
Let  $E =\{  \lambda_1,\dots,\lambda_n \}$ be a finite subset of points in $X$ 
which are separated by $\fA$,
and let $w_1,\dots,w_n$ be complex numbers.  
We seek a necessary and sufficient condition to ensure the existence of a contractive multiplier in $\fA$ which interpolates the given data.
Set $\fJ_E = \{ f \in \fA : f|_E = 0 \}$. 
It is elementary (see Lemma ~\ref{L:dist1}) to verify that 
\begin{align*} 
 \dist(M_f, \fJ_E) \ge  \sup_{L \in \Lat\fA}
 \| P_{L \ominus \fJ_E L} M_f^*|_{L \ominus \fJ_E L}\| .
\end{align*}
Our primary goal here is to find a sufficient condition to ensure that equality holds.  
When such a formula is satisfied, the algebra $\fA$ is said to have an NP family 
of kernels. In this case, there is a corresponding family of Pick matrices for which
simultaneous positivity of its members is a necessary and sufficient condition for interpolation.

Our main result, Theorem~\ref{T:dist_ideal}, states that if $H$ is an reproducing kernel
Hilbert space, and $\fA$ is any unital weak-$*$-closed algebra of multipliers that has 
the strong predual factorization property called property $\bA_1(1)$, 
then $\fA$ admits a Nevan\-linna-Pick family.
A weak-$*$ closed subspace of operators $\fA$ has property $\bA_1(1)$ if every weak-$*$ continuous functional $\phi$
on $\fA$ with $\|\phi\|<1$ can be represented as a vector functional
$\phi(A) = \ip{Ax,y}$ for all $A \in \fA$ with $\|x\|\,\|y\| < 1$.
This property was introduced in \cite{HN82, ABFP}.

While the property $\bA_1(1)$ is very strong, many relevant 
examples of multiplier algebras have it.  
In particular, if the multiplier algebra has $\bA_1(1)$, then the results apply to
all unital weak-$*$ closed subalgebras.
A result of Arias and Popescu \cite{AP95} shows that quotients of the non-commutaitve
Toeplitz algebra have property  $\bA_1(1)$.
We use this to show that the multiplier algebras of all complete NP spaces 
have $\bA_1(1)$.
Thus these results apply in many well known contexts including all subalgebras of the 
multipliers on Drury-Arveson space.
Theorem~\ref{T:dist_ideal} provides an NP theorem for any weak-$*$-closed 
subalgebra of $H^\infty$ acting on Hardy space.
This provides a generalization of the results in \cite{DPRS09} and \cite{Rag09a}.  
In \cite{Rag09b}, Raghupathi shows that Abrahamse's interpolation theorem
is equivalent to a constrained interpolation problem on certain weak-$*$-closed 
subalgebras of $H^\infty$. 
Consequently our distance formula implies Abrahamse's result as well.

Additionally, the above distance formula is still achieved by restricting to the \emph{cyclic} invariant
subspaces of $\fA$.  In certain cases of interest, it suffices to use cyclic vectors
which do not vanish on $E$.  In the case of Hardy space and Drury-Arveson space,
this is accomplished by showing that cyclic subspaces generated by outer functions suffice.
This yields a simplification in the statement of the theorems.
 
The distance formula also formulates the classic Nevan\-linna-Pick 
problem in terms of the Bergman kernel, whose multiplier algebra has a property 
much stronger than $\bA_1(1)$.
Indeed, Bergman space satisfies even a complete distance formula; see Section~\ref{S:matrix}.  
While only of theoretical interest, since interpolation problems for $H^\infty$ 
are better carried out on Hardy space, this result is surprising since Bergman space
is not an NP kernel, and this failure persists even for $2$-point interpolation.

In Section~\ref{S:DA}, algebras of multipliers on complete Nevan\-linna-Pick spaces will be studied.  
The multipliers of any such space are complete quotients of the noncommutative analytic 
Toeplitz algebra $\fL_d$ \cite{DP98b}, which has the property $\bA_1(1)$ \cite{DP98a} and 
even the much stronger property $X_{0,1}$ \cite{Berco98}.
In the proof of Theorem~\ref{T:quotients have A1}, this is used to show that all 
complete NP kernel multiplier algebras inherit $\bA_1(1)$.
Consequently, \emph{any} weak-$*$closed algebra of multipliers on a complete NP space admits an NP family of kernels.

Finally, in Section~\ref{S:matrix}, matrix-valued Nevan\-linna-Pick problems are studied.  
In order to retrieve a suitable NP theorem for matrices of arbitrary size, something much 
stronger than $\bA_1(1)$ is required.  
However, by working with a suitable ampliation of the algebra, it is possible to retrieve 
an NP-type theorem with milder assumptions.  
Both of these results are summarized in Theorem~\ref{T:matrixNP}.  
{}From this, a complete NP family of kernels is presented for the Bergman space 
as well as the classic matrix-valued NP theorem.
A recent result of Ball, Bolotnikov and Ter Horst \cite{BBtH} regarding 
matrix-valued interpolation on the algebra $H^\infty_1$ is generalized to 
arbitrary unital weak-$*$ closed subalgebras of $\Hinf$.

%%%%%%%%%%%%%%%%%%%%%%%%%%%%%%%%%%%%%%%%%%%
\section{Nevan\-linna-Pick Families}\label{S:NPfamilies}

Let $H$ denote a reproducing kernel Hilbert space of complex-valued 
functions on a set $X$, with $k_\lambda$ as the reproducing kernel at $\lambda \in X$. 
Let $k(\lambda,\mu) := \ip{ k_\mu, k_\lambda }$ denote the associated 
positive kernel on $ X \times X$.  
We define the \emph{multiplier algebra of $H$} by
\begin{align*}
 \fM(H) = \{ f: X \to \bC : fh \in H \FORAL h \in H  \},
\end{align*}
where the product $fg$ is defined pointwise.
Each multiplier $f$ on $H$ defines the corresponding multiplication operator 
$M_f$, which is bounded by an application of the closed graph theorem.  
We may then regard $\fM(H)$ as an abelian subalgebra of $\B(H)$.
It is well known that an operator $T \in \B(H)$ defines a multiplication operator 
on $H$ if and only if each kernel function $k_\lambda$ is an eigenvector for $T^*$.  
{}From this it easily follows that $\fM(H)$ is closed in the weak operator topology.

We say that a \textit{unital} subalgebra $\fA$ of $\B(H)$ is a \emph{dual algebra} 
if it is closed in the weak-$*$ topology on $\B(H)$.
If $\fA$ is contained in the multiplier algebra of $H$, we say that $\fA$ is 
a \emph{dual algebra of multipliers of $H$}.
Suppose that $E$ is a finite subset of $X$.  Let $\fJ_E $ denote 
the weak-$*$ closed ideal of multipliers in $\fA$ that vanish on $E$.
When $E=\{\lambda\}$, we write $\fJ_\lambda$.
If the context is clear, we may write $\fJ := \fJ_E$.  

Suppose $L$ is an invariant subspace of $\fA$ and let $P_L$ denote the orthogonal projection of $H$ onto $L$. Then $L$ is also a 
reproducing kernel Hilbert space on the points in $X$ which are not annihilated by $L$. The reproducing 
kernel on this set is given by $P_L k_\lambda$.  
Every $f \in \fA$ defines a multiplier on $L$ since
\[
 (M_f|_L)^* P_Lk_\lambda= P_L M_f^* P_Lk_\lambda = P_L M_f^* k_\lambda = 
 \ol{f(\lambda)}P_L k_\lambda .
\]
The following lemma shows that, in certain cases, 
it is possible to extend this kernel to all of $X$.

%%%%%%%%%%%%%%%%%%%%
\begin{lem} \label{L:extend kL}
Suppose $\fA$ is a dual algebra of multipliers on $H$.  
If $L = \fA[h]$ is a cyclic invariant subspace of $\fA$, 
it is possible to extend the kernel function $P_Lk_\lambda$
$($which is non-zero on $\{\lambda : h(\lambda) \ne 0 \})$ to a kernel $k^L_\lambda$ defined
on all of $X$ so that $k^L_\lambda = 0$ only when $\fJ_\lambda[h]=\fA[h]$. 
For each $f \in \fM(H)$, $k^L_\lambda$ satisfies the fundamental relation
\[ P_L M_f^* k^L_\lambda = \ol{f(\lambda)} k^L_\lambda ,  \]
and thus
\[
 \ip{ M_f k^L_{\lambda},  k^L_{\lambda}} = f(\lambda) \| k^L_{\lambda} \|^2 
 \qforal \lambda \in X.
\] 
\end{lem}

\begin{proof}
Since $L = \fA[h]$ is a cyclic subspace and $\dim \fA/\fJ_\lambda = 1$, 
it follows that $\dim \fA[h]/\fJ_\lambda[h] \le 1$.
If $P_L k_\lambda \ne 0$, then this is an eigenvector for $P_L \fA^*$
as shown above.  
For $f \in \fJ_\lambda$, 
\[ 
 \ip{ k^L_\lambda, fh} = \ip{ P_L M_f^* k^L_\lambda, h} = \ip{\ol{f(\lambda)} k^L_\lambda,h} = 0 .
\]
So $P_Lk_\lambda$ belongs to $\fA[h] \ominus \fJ_\lambda[h]$
and we set $k^L_\lambda = P_L k_\lambda$.

When $P_L k_\lambda = 0$ but $\dim \fA[h]/\fJ_\lambda[h] = 1$, 
we pick a unit vector $k^L_\lambda$ in $\dim \fA[h] \ominus \fJ_\lambda[h]$.  
Then for $f \in \fM(H)$, $f-f(\lambda) \one$ lies in $\fJ_\lambda$.  
Hence
\[
 \ip{ M_f k^L_\lambda,  k^L_\lambda} =
 f(\lambda) \ip{ k^L_\lambda,  k^L_\lambda} - 
 \ip{ (f-f(\lambda) \one) k^L_\lambda,  k^L_\lambda} =
 f(\lambda) .
\]
Also, $\fJ_\lambda[h] \in \Lat(\fA)$ and thus $P_L M_f^* k^L_\lambda$
lies in $L \ominus \fJ_\lambda[h] = \bC k^L_\lambda$.
The previous computation shows that 
\[
 P_L M_f^* k^L_\lambda = \ol{f(\lambda)} k^L_\lambda .  
\]
This extends the kernel $k^L$ to all of $X$.
\end{proof}

This kernel allows us to evaluate the multipliers at a generally much
larger subset of $X$ (see Example~\ref{E:Bergman1} below) than
just using $P_L k_\lambda$.
However some continuity is lost for evaluation of functions in $L$.
Since we are primarily interested in interpolation questions about 
the multiplier algebra, evaluation of the multipliers is more important.

%%%%%%%%%%%%%%%%%%%%
\begin{defn}
For any dual algebra $\fA$ of multipliers on $H$ and any cyclic 
invariant subspace $L \in \CycLat(\fA)$, 
let $k^L_\lambda$ denote the extended
reproducing kernel on $L$ constructed in Lemma~\ref{L:extend kL} at the point $\lambda$.
\end{defn}

%%%%%%%%%%%%%%%%%%%%
\begin{eg} \label{E:Bergman1}
In spaces of analytic functions, it is often possible to fully describe 
the kernel structure on $L \in \CycLat(\fA)$.  
Indeed, suppose that $H = A^2(\bD)$ is the Bergman space, 
and $\fA = \Hinf$.
Let $L = \Hinf[h]$ for some non-zero function $h \in A^2(\bD)$.  
Then 
\[
 X_L = \{ \lambda \in \bD : P_L k_\lambda \ne 0 \} 
        = \{ \lambda \in \bD : h(\lambda) \ne 0 \} .
\]
However, since the Bergman space consists of analytic functions,
$h$ vanishes only to some finite order on each of its zeros.

It is routine to verify that for each $n \geq 0$ and $\lambda \in \bD$, 
there is a function $k_{\lambda,n} \in A^2(\bD)$ such that 
\[ \ip{ h, k_{\lambda,n} } = h^{(n)}(\lambda) \qfor h \in A^2(\bD) .\]  
Suppose that  $h$ vanishes at $\lambda$ with multiplicity $r \ge 0$.  
We claim that $P_L k_{\lambda,r} \ne 0$ and
$P_L k_{\lambda,n} = 0$ for $0 \le n < r$. 
Indeed, for any $f \in \Hinf$ and $n \le r$,
\begin{align*}
 \ip{f h, k_{\lambda,n}} &= (f h)^{(r)}(\lambda_i) \\&
 = \sum_{j=0}^r {r\choose j} f^{(j)}(\lambda) h^{(r-j)}(\lambda) \\&
 = \begin{cases} 0 &\qif 0 \le n<r\\
    f(\lambda) h^{(r)}(\lambda) &\qif n=r . \end{cases}
\end{align*}
So $P_L k_{\lambda,n} = 0$ for $0 \le n < r$.
Set $k^L_\lambda = P_L k_{\lambda,r}/\|P_L k_{\lambda,r}\|$.
This calculation shows that if $f\in\fJ_\lambda$, then $\ip{fh,k^L_\lambda} = 0$.
So $k^L_\lambda$ belongs to $\fA[h]\ominus\fJ_\lambda[h]$.
Now for $f,g \in \Hinf$,
\begin{align*}
 \ip{f h, M_g^* k^L_\lambda} &=
 \ip{gf h, k^L_\lambda} \\&=
 g(\lambda) f(\lambda) h^{(r)}(\lambda_i) =
 \ip{f h, \ol{g(\lambda)} k^L_\lambda}.
\end{align*}
It follows that 
\[
 (M_g|_L)^* k^L_\lambda = P_L M_g^* k^L_\lambda = \ol{g(\lambda)} k^L_\lambda .
\]
Thus $g$ is a multiplier for this reproducing kernel.

An identical construction is possible for any space of analytic functions on
the unit disk in which the Taylor coefficients of the power series expansion
about $0$ are continuous and composition by the M\"obius automorphisms 
of the  disk are bounded maps.  
\end{eg}

%%%%%%%%%%%%%%%%%%%%
\begin{rem} \label{R:why cyclic}
The Bergman space is also a good place to illustrate 
why we require cyclic invariant subspaces.
The Bergman shift $B$ is a universal dilator for strict contractions \cite{ABFP}.
For example,  fix a point $\lambda \in \bD$.
Then $B$ has an invariant subspace $L$ such that $N = L \ominus \fJ_\lambda L$
is infinite dimensional.  
Since $(M_f|_L)^*|_N = \ol{f(\lambda)} I_N$, there is no canonical choice for $k^L_\lambda$.
Similarly, for any $0 < r < 1$, we can obtain the infinite ampliation $rB^{(\infty)}$
as the compression of $B$ to a semi-invariant subspace $M$, and
this has the same issue for every $|\lambda|<r$.
\medbreak

On the other hand, we can always identify a kernel structure on any invariant
subspace $L \in \Lat \fA$ if we allow multiplicity.
The subspaces $N_\lambda = L \ominus \fJ_\lambda L$ satisfy
$P_{N_\lambda} M_f^*|_{N_\lambda} = \ol{f(\lambda)} P_{N_\lambda}$.
So if $k$ is any unit vector in $N_L$, we obtain
\[ \ip{ M_f k,  k} = f(\lambda) \qforal f \in \fA .\]
The spaces $\{N_\lambda : \lambda \in X \}$ are linearly independent
and together they span $L$.
See the continued discussion later in Remark~\ref{R:multiplicity}.
\end{rem}

Our primary interest will be Nevan\-linna-Pick interpolation on some finite 
subset $E = \{\lambda_1,...,\lambda_n)$ of $X$ by functions in the algebra $\fA$.  
It could be the case that $\fA$ fails to separate certain points in $X$, and so 
we impose the natural constraint that $E$ contains 
at most one representative from any set of points that $\fA$ identifies.  
It follows that the kernels $k_{\lambda_i}$ form a linearly independent set.  
Indeed, since $\fA$ separates these points, we can find elements 
$p_1,...,p_n \in \fA$ such that $p_i(\lambda_j) = \delta_{ij}$.  
Hence if $\sum_{i=1}^n \alpha_i k_{\lambda_i} = 0$, we find that 
\[
 0 = M_{p_i}^* \Big( \sum_{i=1}^n \alpha_i k_{\lambda_i} \Big) = \alpha_i k_{\lambda_i}
 \qfor 1 \le i \le n.
\]  
The quotient algebra $\fA / \fJ_E$ is $n$-dimensional, 
and is spanned by the idempotents $\{ p_i + \fJ : 1 \le i \le n \}$.  
We seek to establish useful formulae for the norm on $\fA / \fJ$.  
These so-called operator algebras of idempotents have been 
studied by Paulsen in \cite{Pau01}.

%%%%%%%%%%%%%%%%%%%%
\begin{defn}
Given a finite subset $E$ of $X$,   set 
\[M =  M(E) = \spn\{ k_\lambda : \lambda \in E \} .\]
For $L \in \Lat(\fA)$, define 
\[ M_L =P_L M \qand N_L = L \ominus \fJ_E L . \]
\end{defn}

Observe that if $L \in \Lat \fA$, then $\ol{\fJ_E L}$ is also invariant and
is contained in $L$.  
Thus $N_L$ is a semi-invariant subspace.  In particular, 
$P_{N_L} M_f P_L = P_{N_L} M_f P_{N_L}$ and 
compression to $N_L$ is a contractive homomorphism.
Likewise, $M(E)$ is co-invariant, and so $\ol{M(E)+L^\perp}$ is co-invariant.
Observe that 
\[ M_L = \ol{(M(E)+L^\perp)} \cap L = L \ominus (M(E)+L^\perp)^\perp,\]
and so it is also semi-invariant.

%%%%%%%%%%%%%%%%%%%%
\begin{lem}
Given a finite subset $E \subset X$ on which $\fA$ separates points,
and $L = \fA[h]$ in $\CycLat(\fA)$,
the space $N_L$ is a reproducing kernel Hilbert space over $E$
with kernel $\{ k^L_\lambda : \lambda \in E \}$; and the non-zero
elements of this set form a basis for $N_L$.  

Also $M_L$ is a subspace of $N_L$ spanned by
\[ \{ k^L_\lambda = P_L k_\lambda : \lambda \in E,\ h(\lambda) \ne 0 \} ,\] 
and it is a reproducing kernel Hilbert space over $\{\lambda \in E : h(\lambda) \ne 0 \}$.
\end{lem} 

\begin{proof}
For each $\lambda \in E$,
\[
 k^L_\lambda \in L \ominus \fJ_\lambda[h] \subset L \ominus \fJ_E[h] = N_L .
\]
Let $E_L = \{\lambda \in E : k^L_\lambda \ne 0 \}$.
Then for $\lambda\in E\bsl E_L$, $L = \ol{\fJ_\lambda L}$.
Since $\fJ_{E_L} = \prod_{\lambda \in E\bsl E_L} \fJ_\lambda$, we
see that $\ol{\fJ_{E_L} L} = L$.
Hence we may factor $\fJ_E = \fJ_{E_L} \fJ_{E\bsl E_L}$ and note that
\[ \ol{\fJ_E L} = \ol{ \fJ_{E_L} \ol{ \fJ_{E\bsl E_L} L}} =  \ol{ \fJ_{E_L} L}. \]
Now $\dim \fA/\fJ_{E_L} = |E_L|$, so $\dim N_L \le |E_L|$.
But $N_L$ contains the non-zero vectors $k^L_\lambda$ for 
$\lambda\in E_L$.  For $f\in\fA$ and $\lambda\in E_L$,
\begin{align*}
 P_{N_L} M_f^* k^L_\lambda &= P_L P_{\fJ_E[h]}^\perp M_f^* k^L_\lambda %\\&
 = P_L M_f^* k^L_\lambda = \ol{f(\lambda)} k^L_\lambda .
\end{align*} 
Because $\fA$ separates the points of $E_L$, it follows that these vectors
are eigenfunctions for distinct characters of $\fA$, and thus are linearly independent.
This set has the same cardinality as $\dim N_L$, and therefore it forms a basis.

Now $M_L$ is spanned by $\{P_L k^\lambda : \lambda \in E\}$,
and it suffices to use the non-zero elements.  These coincide
with $k^L_\lambda$ on $E_L^0 := \{\lambda \in E : h(\lambda) \ne 0 \}$.
This is a subset of the basis for $N_L$, and hence $M_L$ is a subspace of $N_L$. 
It is also a reproducing kernel space on $E_L^0$.
\end{proof}

The equality $M_L = N_L$ holds in the following important case.
The proof is immediate from the lemma.

%%%%%%%%%%%%%%%%%%%%
\begin{cor} \label{C:cyclic ML}
Suppose that $E$ is a finite subset of $X$ on which $\fA$ separates points,
and $L = \fA[h]$ in $\CycLat(\fA)$. 
If $h$ does not vanish on $E$, then $M_L = N_L$.
\end{cor}

We will be searching for appropriate distance formulae from elements of $\fA$ 
to an ideal $\fJ_E$ in order to obtain Nevan\-linna-Pick type results.
We note certain easy estimates which always hold.

%%%%%%%%%%%%%%%%%%%%
\begin{lem} \label{L:dist1}
Suppose that $\fA$ is a dual algebra on $H$,
and let $\fJ$ be a \wot-closed ideal. 
Set $N_L = L \ominus \fJ L$ and $M_L = P_L M(E)$ for $L \in \Lat \fA$.
Then the following distance estimates hold:
\begin{align*}
 \dist(f, \fJ) &\ge \sup_{L \in \Lat(\fA)} \| P_{N_L} M_f P_{N_L}  \|
 = \sup_{L \in \CycLat(\fA)} \| P_{N_L} M_f P_{N_L}  \|  \\&
 \ge  \sup_{L \in \CycLat(\fA)} \| P_{M_L} M_f P_{M_L}  \|  
 = \sup_{L \in \Lat(\fA)} \| P_{M_L} M_f P_{M_L} \|.
\end{align*}
\end{lem}

\begin{proof}
Suppose $L$ is an invariant subspace of $\fA$. 
For $M_f \in \fA$ and $M_g \in \fJ$, compute
\begin{align*}
 \| M_f - M_g \| &\ge  \| P_{N_L} (M_f - M_g) P_{N_L}  \|  \\&
= \| P_{N_L} M_f P_{N_L} \| % \\&
 \ge \| P_{M_L} M_f P_{M_L} \|.
\end{align*}
Taking an infimum over $g\in \fJ$ and a supremum over $\Lat(\fA)$, 
we obtain 
\[
\dist(f, \fJ) \ge \sup_{L \in \Lat(\fA)} \| P_{N_L} M_fP_{N_L}  \|
\ge  \sup_{L \in \CycLat(\fA)} \| P_{N_L} M_f P_{N_L}  \|  . 
\]
Since $M_L$ is contained in $N_L$, we have
\[
 \sup_{L \in \CycLat(\fA)} \| P_{N_L} M_f P_{N_L}  \|    
 \ge  \sup_{L \in \CycLat(\fA)} \| P_{M_L} M_f P_{M_L} \| .
\]

Now consider an arbitrary element $L \in \Lat(\fA)$.
Then 
\begin{align*}
 \| P_{N_L} M_f P_{N_L} \| &= \| P_{N_L} M_f P_L \| =
 \sup_{\|h\|=1,\, h\in L}  \| P_{N_L} M_f P_L h\| \\&
 = \sup_{\|h\|=1,\, h\in L}  \| (P_L - P_{\fJ_E L}) M_f P_{\fA[h]} h \| \\&
 \le \sup_{\|h\|=1,\, h\in L}  \|  (P_L - P_{\fJ_E \fA[h]}) P_{N_{\fA[h]}} M_f P_{\fA[h]} h \| \\&
 = \sup_{\|h\|=1,\, h\in L}  \| P_{N_{\fA[h]}} M_f P_{\fA[h]} h \| \\&
 \le   \sup_{L \in \CycLat(\fA)} \| P_{N_L} M_f P_L \| 
 = \sup_{L \in \CycLat(\fA)} \| P_{N_L} M_f P_{N_L}  \|  .
\end{align*}
Similarly,
\begin{align*}
 \| P_{M_L} M_f P_{M_L} \| &= \| P_{M_L} M_f P_L \| 
 = \sup_{\|h\|=1,\, h\in L}  \| P_{M_L} M_f P_L h\| \\&
 = \sup_{\|h\|=1,\, h\in L}  \| P_{M_L} M_f P_{\fA[h]} h \| \\&
 = \sup_{\|h\|=1,\, h\in L}  \| P_{M_L} P_{N_{\fA[h]}} M_f P_{\fA[h]} h \| \\&
 = \sup_{\|h\|=1,\, h\in L}  \| P_{M_{\fA[h]}} M_f P_{\fA[h]} h \| \\&
 \le   \sup_{L \in \CycLat(\fA)} \| P_{M_L} M_f P_L \| 
 = \sup_{L \in \CycLat(\fA)} \| P_{M_L} M_f P_{M_L} \| . 
\end{align*}
So we may take the supremum over all invariant subspaces without
changing these two  distance estimates.
\end{proof} 

When $\fA$ is a dual algebra of multipliers, 
it is convenient to work with the adjoints due to 
their rich collection of eigenvectors.  
Lemma~\ref{L:dist1} gives us
\begin{align*}
 \dist(f, \fJ) \geq  \sup_{L \in \Lat(\fA)}\| P_L M_f^*|_{N_L}\|
 = \sup_{L \in \CycLat(\fA)}\| P_{N_L} M_f^*|_{N_L}\|.
\end{align*}

\begin{rem} \label{R:multiplicity}
This proposition shows that it suffices to look at cyclic subspaces.
In view of Lemma~\ref{L:extend kL}, this is of particular importance
when dealing with algebras of multipliers.
But in fact, there is little additional complication when the subspaces 
$N_\lambda = L \ominus \fJ_\lambda L$ have dimension greater than one.

These subspaces are at a positive angle to each other even when
they are infinite dimensional because the restriction of $P_L M_f^*$
to $N_\lambda$ is just $\ol{f(\lambda)} P_{N_\lambda}$.  
When $\fA$ separates points $\lambda$ and $\mu$, the boundedness
of $M_f^*$ yields a positive angle between eigenspaces.
Moreover the spaces $N_\lambda$ for $\lambda \in X$ span $L$.

We are interested in the norm $\| P_{N_L} M_f^*|_{N_L}\|$.
This is approximately achieved at a vector $h = \sum x_\lambda$
where this is a finite sum of vectors $x_\lambda \in N_\lambda$.
Since 
\[ P_L M_f^* h = \sum \ol{f(\lambda)} x_\lambda ,\]
it follows that $K = \spn\{ x_\lambda : \lambda \in X \}$ is invariant
for $P_L M_f^*$ for all $M_f \in \fA$.
In particular, we obtain that $\| P_{N_L} M_f^*|_{N_L}\| \le 1$
if and only if $\| P_{K} M_f^*|_{K}\| \le 1$ for each subspace $K$
of the form just described.  This is equivalent to saying 
$P_K - P_K M_f M_f^*|_K \ge 0$  because of semi-invariance.
Because the $x_\lambda$ span $K$, this occurs if and only if 
\[ 
 0\le  \Big[ \ip{(I - M_f M_f^*) x_{\lambda_j}, x_{\lambda_i}} \Big]
  = \Big[ \big(1- f(\lambda_i)\ol{f(\lambda_j)}\big) \ip{x_{\lambda_j}, x_{\lambda_i}} \Big] .
\]
Thus, the norm condition is equivalent to the simultaneous positivity 
of a family of Pick matrices.

Moreover, in the case of $\fJ = \fJ_E$ for a finite set $E=\{\lambda_1,\dots,\lambda_n \}$
which is separated by $\fA$, this family of Pick matrices is positive
if and only if we have positivity of the operator matrix
\[
  \Big[ \big( 1- f(\lambda_i)\ol{f(\lambda_j)} \big) 
  P_{N_{\lambda_i}}P_{N_{\lambda_j}}  \Big]_{n\times n} .
\]
For an arbitrary ideal $\fJ$, we can take the supremum over all
finite subsets $E$ of $X$.
\end{rem}

\begin{defn}
The collection $\{k^{L} :  L \in \CycLat(\fA) \}$ is said to 
be a \emph{Nevan\-linna-Pick family of kernels for $\fA$} 
if for every finite subset $E$ of $X$ and every $f \in \fA$,
(and $N_L = L \ominus \fJ_E L$)
\[ 
 \dist(f, \fJ_E) =  \sup_{L \in \CycLat(\fA)}\| P_{N_L} M_f^*|_{N_L}\| .
\]
\end{defn}

The following routine theorem reveals why Nevan\-linna-Pick 
families are given their name.

%%%%%%%%%%%%%%%%%%%%
\begin{thm} \label{T:NPfamily}
Let $\fA$ be a dual algebra of multipliers on a Hilbert space $H$. 
The family $\{k^L :  L \in \CycLat(\fA) \}$
is a Nevan\-linna-Pick family of kernels for $\fA$ if and only if 
the following statement holds:\\[.5ex]
Given $E = \{\lambda_1, \dots, \lambda_n\}$ distinct points in $X$
which are separated by $\fA$, 
and complex scalars $w_1,\dots,w_n$, 
there is a multiplier $f$ in the unit ball of $\fA$  such that  
$f(\lambda_i) = w_i$  for  $1 \le i \le n$ if and only if the Pick matrices
\begin{align*}
\Big[
 (1-w_i\ol{w_j}) k^L(\lambda_i,\lambda_j)
\Big]_{n \times n}
\end{align*}
are positive definite for every $L \in \CycLat(\fA)$.
\end{thm}

\begin{proof}
Suppose that $\{k^L :  L \in \CycLat(\fA) \}$ is a Nevalinna-Pick family.
If such a multiplier $f$ exists, then the positivity of the matrices 
follows from standard results in reproducing kernel Hilbert spaces 
(see \cite{AMc02}, for instance).  

On the other hand, suppose that all such matrices are positive definite.  
Since $\fA$ separates the points in $E$,
find an arbitrary interpolant $p \in \fA$ so that $p(\lambda_i) = w_i$
(for example,  consider the functions $p_i$ as defined earlier,
and let $p = \sum_{i=1}^n w_ip_i$). 
It is a routine argument in this theory that $\| P_{N_L} M^*_p |_{ N_L} \| \le 1$
if and only if
\[ I - P_{N_L} M_p P_{N_L} M^*_p P_{N_L} \ge 0 .\]
Since $k^L_{\lambda_i}$ spans $N_L$, this holds if and only if
\begin{align*}
 0 &\le 
 \Big[ (I - P_{N_L} M_p P_{N_L} M^*_p P_{N_L}) 
 k^L_{\lambda_j}, k^L_{\lambda_i} \Big] \\ &=
 \Big[ k^L_{\lambda_j}, k^L_{\lambda_i} \Big] -
 \Big[ M^*_p k^L_{\lambda_j}, M_p^* k^L_{\lambda_i} \Big] \\ &=
 \Big[ k^L_{\lambda_j}, k^L_{\lambda_i} \Big] -
 \Big[ \ol{w_j} k^L_{\lambda_j}, \ol{w_i} k^L_{\lambda_i} \Big] \\ &=
 \Big[ (1-w_i\ol{w_j}) k^L(\lambda_i,\lambda_j) \Big].
\end{align*}

It follows that 
\begin{align*}
 \dist(p, \fJ_E) = \sup_{L \in \CycLat(\fA)}\| P_L M^*_p |_{N_L} \| \le 1 .
\end{align*}
The ideal $\fJ_E$ is weak-$*$-closed, and so the distance is attained at some $g \in \fJ$.  
The multiplier $f := p - g$ is in the unit ball of $\fA$ and interpolates the given data. 

Conversely, suppose that the interpolation property holds.
Fix $f \in \fA$ such that $\sup_{L \in \CycLat(\fA)}\| P_L M_f^*|_{N_L} \| = 1$.  
By assumption, there is a multiplier $g$ in the unit ball of $\fA$ 
such that $g(\lambda_i) = f(\lambda_i)$ for each $1 \le i \le n$.  
Hence
\[
 \dist(f,\fJ_E) \le \| f - (f-g) \| = \|g\| \le 1 .
\]
By Lemma~\ref{L:dist1}, $\dist(M_f,\fA) \geq \sup_{L \in \Lat(\fA)}\| P_L M_f^*|_{N_L} \| = 1$.
Therefore the family $ \{k^L :  L \in \CycLat(\fA) \}$
is a Nevan\-linna-Pick family of kernels for $\fA$.
\end{proof}

We are also interested when the subspaces $M_L$ suffice to compute the distance.
The same argument shows that these spaces form a Nevan\-linna-Pick
family for interpolation on the set $E$.
This will occur if we can show that cyclic subspaces generated by 
elements $h$ that do not vanish on $E$ suffice in the distance estimate.
This is not always the case, but it does happen in important special cases.

%%%%%%%%%%%%%%%%%%%%%%%%%%%%%%%
\section{Algebras with property $\bA_1(1)$} \label{S:A1}

Let $\fA$ be a dual algebra in $\B(H)$.
Given vectors $x$ and $y$  in $H$,
let $[xy^*]$ denote the vector functional $A \to \ip{Ax,y}$ for $A \in \fA$.  
Every weak-$*$ continuous functional on $\fA$ is a
(generally infinite) linear combination of these vector functionals. 

One says that $\fA$ has property $\bA_1(r)$ if, 
for each weak-$*$ continuous functional $\phi$ on $\fA$ with $\|\phi\| < 1$ , 
there are vectors $x$ and $y$ so that $\phi = [xy^*]$ and 
$\| x \|\, \| y \| < r$. 
Property $\bA_1(r)$ implies that the weak-$*$ 
and weak operator topologies of $\fA$ coincide, since the 
weak operator continuous linearly functionals are precisely 
the finite linear combinations of vector functionals.

It is well known that for any reproducing kernel Hilbert space,
the multiplier algebra $\fM(H)$ is reflexive. 
This is because $\bC k_\lambda$ is invariant for $\fM(H)^*$
for each kernel function $k_\lambda$. 
Hence, if $T \in \Alg(\Lat \fM(H))$, then $k_\lambda$ is an 
eigenvector for $T^*$. 
It follows that $T$ is a multiplier on $H$ by standard results.

Suppose $\fA$ is a dual subalgebra of $\fM(H)$  that is relatively reflexive;  
that is, the algebra satisfies
\begin{align*}
 \fA = \Alg(\Lat\fA) \cap \fM(H).
\end{align*}
The reflexivity of the whole multiplier algebra
implies that $\fA$ is actually reflexive.

We saw that a multiplier $f$ in $\fA$ defines the multiplication 
operator $M_f|_L$ on every invariant subspace $L$ of $\fA$.  
On the other hand, if $f$ is simultaneously a multiplier on each $L$ in $\Lat\fA$, 
then it is clearly in $\fM(H)$ and leaves every $L$ invariant.  
Consequently, $\fA$ is the \emph{largest} algebra of multipliers 
for the family of subspaces $\Lat(\fA)$.  
It is clear that it suffices to consider cyclic 
invariant subspaces.  So we have
\[
 \fA = \bigcap_{L \in \Lat(\fA)} \fM(L) \ \ = \bigcap_{L \in \CycLat(\fA)} \fM(L).
\]
Following the notation of Agler and McCarthy \cite{AMc02}, 
$\CycLat\fA$ is called a \emph{realizable} collection of reproducing 
kernel Hilbert spaces.

The first part of the following theorem is contained in
Hadwin and Nordgren's paper \cite[Prop.~2.5(1)]{HN82} on
$\bA_1(r)$ (which is called $D_\sigma(r)$ there).  
The second part is from Kraus and Larson \cite[Theorem~3.3]{KL86}.
We say that a subalgebra $\fA$ of $\fB$ is \textit{relatively hyper-reflexive}
if there is a constant $C$ so that
\[
 \dist(B,\fA) \le C \sup_{L\in\Lat \fA} \| P_L^\perp B P_L \| 
 \qforal B \in \fB .
\]
The optimal value of $C$ is the relative hyper-reflexivity constant.
Again, it is clear that cyclic subspaces suffice.

%%%%%%%%%%%%%%%%%%%%
\begin{thm}[Hadwin-Nordgren, Kraus-Larson] \label{T: HN-KL}
Suppose $\fB$ is a dual subalgebra of $\fM(H)$ 
and has property $\bA_1(r)$.  
Then every \wot-closed unital subalgebra $\fA$ of $\fB$ is reflexive.  
Moreover, $\fA$ is relatively hyper-reflexive in $\fB$
with distance constant at most $r$.
\end{thm}

If $\fB$ has property $\bA_1(1)$, we obtain an exact distance formula.

%%%%%%%%%%%%%%%%%%%%
\begin{cor} \label{C:hyperA1}
Suppose that $\fB$ has property $\bA_1(1)$, 
and let $\fA$ be a \wot-closed unital subalgebra.
Then
\[
 \dist(B, \fA) = \sup_{L \in \CycLat(\fA)}\| P^\perp_L B P_L  \| \qforal B \in \fB .
\]
\end{cor}

We can use the same methods to obtain a distance formula to any 
\mbox{weak-$*$} closed ideal.
We do not have a reference, so we supply a proof.  The ideas go back
to the seminal work of Sarason \cite{Sar67}.  
An argument similar to this one is contained in the proof of \cite[Theorem~2.1]{DP98b}.

%%%%%%%%%%%%%%%%%%%%
\begin{thm} \label{T:dist_ideal}
Suppose that $\fA$ is a dual algebra on $H$ with property $\bA_1(1)$,
and let $\fJ$ be any \wot-closed ideal of $\fA$.
Then we obtain
\[
 \dist(A, \fJ) =  \sup_{L \in \CycLat(\fA)}\| P_{L \ominus \fJ L} M_f^*|_{L \ominus \fJ L}\| .
\]
\end{thm}

\begin{proof}
By Lemma~\ref{L:dist1}, $\dist(A,\fJ)$ dominates the right hand side.

Conversely, given $A \in \fA$ and $\ep>0$, pick $\phi \in \fA_*$ such that
$\phi|_\fJ = 0$, $\|\phi\| < 1 + \ep$ and $| \phi(A) | = \dist(A,\fJ)$.
Using property $\bA_1(1)$, we obtain vectors $x$ and $y$
with $\|x\|=1$ and $\|y\| < 1 + \ep$ so that $\phi = [xy^*]$.
Set $L = \fA[x]$ in $\CycLat \fA$. Since $L$ is invariant, 
we can and do replace $y$ by $P_L y$ without changing $[xy^*]$.
Since $\phi|_\fJ = 0$, $y$ is orthogonal to $\fJ[x] = \fJ L$.
Hence $y$ belongs to $L \ominus \fJ L$. 
Therefore
\begin{align*}
  \dist(A, \fJ) &= | \ip{Ax,y} | = 
  | \ip{ A P_L x, P_{N_L}y} | \\&=
  | \ip{ P_{L \ominus \fJ L}A P_L x, y}  | =  | \ip{ P_{L \ominus \fJ L} A P_{L \ominus \fJ L} x,y} | \\&\le
  \| P_{L \ominus \fJ L} A P_{L \ominus \fJ L} \| \, \|x\|\, \|y\| \\&< 
  (1+\ep) \| P_{L \ominus \fJ L} A|_{L \ominus \fJ L} \| .
\end{align*}
It follows that the right hand side dominates $\dist(A,\fJ)$.
\end{proof}

We apply this to the context of interpolation in reproducing kernel Hilbert spaces.

%%%%%%%%%%%%%%%%%%%%
\begin{thm} \label{T:A1NP}
Suppose $\fA$ is a dual algebra of multipliers on $H$ with property $\bA_1(1)$.
Let $E$ be a finite subset of $X$ which is separated by $\fA$.
Then the following distance formula holds:
\[
 \dist(f, \fJ_E) =  \sup \{\| P_{N_L} M^*_f |_{N_L} \| :  L \in\CycLat(\fA) \}  .
\]
That is, $\{k^L :  L \in \CycLat(\fA) \}$ is a  
Nevan\-linna-Pick family of kernels for $\fA$.  
\end{thm}

%%%%%%%%%%%%%%%%%%%%
\begin{cor} \label{C:A1NP}
Suppose $\fA$ is a dual algebra of multipliers on a reproducing 
kernel Hilbert space $H$ on $X$ that has property $\bA_1(1)$.  
Let  $\{\lambda_1,...,\lambda_n\}$ be distinct points in $X$ 
separated by $\fA$ and let $\{w_1,...,w_n\}$ be complex numbers. 
There is a multiplier $f$ in the unit ball of $\fA$ such that 
$f(\lambda_i) = w_i$ for each $i$ if and only if 
\begin{align*}
 \Big[ (1 - w_i\ol{w_j})k^L(\lambda_i,\lambda_j) \Big] \geq 0
 \qforal L \in \CycLat(\fA).
\end{align*}
\end{cor}

By replacing the property $\bA_1(1)$ with $\bA_1(r)$ for $r > 1$, 
one may repeat the proof of the previous theorem and obtain an 
analogous result where the norm of the interpolant is bounded by $r$.

In applications, is it convenient to use the spaces $M_L$ instead of $N_L$.
In light of Corollary~\ref{C:cyclic ML}, this is possible for cyclic
subspaces $L=\fA[h]$ provided that $h$ does not vanish on $E$.
A refinement of the $\bA_1(1)$ property can make this possible.
The distinction made in this theorem is that the kernel is formed
as simply the compressions $P_L k_{\lambda}$.

\begin{thm} \label{T:A1ML}
Let $\fA$ be a dual algebra of multipliers on $H$,
and let $E= \{\lambda_1,...,\lambda_n\}$ be a finite 
subset of $X$ which is separated by $\fA$.
Suppose that $\fA$ has property $\bA_1(1)$ with the additional 
stipulation that each weak-$*$ continuous functional $\phi$ 
on $\fA$  with $\|\phi\|<1$ can be factored as $\phi = [xy^*]$ with 
$\|x\|\,\|y\| <1$ such that $x$ does not vanish on $E$.  
Then there is a multiplier in the unit ball of $\fA$ with $f(\lambda_i) = w_i$
for $\lambda_i \in E$ if and only if
\begin{align*}
 \Big[ (1 - w_i\ol{w_j}) \ip{P_L k_{\lambda_j},k_{\lambda_i}} \Big] \geq 0
\end{align*}
for all cyclic subspaces $L= \fA[h]$ where $h$ does not vanish on $E$.
\end{thm}

%%%%%%%%%%%%%%%%%%%%%%%%%%%%%%%%%%
\section{Subalgebras of $\Hinf$} \label{S:single}

In this section, we discuss the algebra $\Hinf$ and its subalgebras.
In particular, we show how to recover results of Davidson, Paulsen, Raghupathi
and Singh \cite{DPRS09} and Raghupathi \cite{Rag09a}.
We also show how one can formulate a Nevan\-linna-Pick theorem for $\Hinf$
on Bergman space.  While this is only of theoretical interest,
since interpolation is better done on the Hardy space, this has been
an open question for some time.

%%%%%%%%%%%%%%%%%%%%
\subsection*{Hardy space}
Let $H = H^2$ be Hardy space, so that $\fM(H)=\Hinf$.
It is well known that this algebra has property $\bA_1(1)$  \cite[Theorem 3.7]{BFP83}.

Moreover, any weak-$*$ continuous functional $\phi$ on $\Hinf$
can be factored exactly as $\phi = [xy^*]$ with $\|x\|\,\|y\| = \|\phi\|$.
Now if the inner-outer factorization of $x$ is $x=\omega h$, with $h$ outer
and $\omega$ inner, then $\phi = [h k^*]$ where $k = M_\omega^*y$.
Thus the factorization may be chosen so that $h$  is outer,
and in particular, does not vanish on $\bD$.

Consider the classical Nevan\-linna-Pick problem for a finite
subset $E$ of $\bD$.
Following Sarason's approach \cite{Sar67},  let  $B_E$ denoting the 
finite Blaschke product with simple zeros at $E$.
Then $\fJ_E = B_E\Hinf$.
Beurling's invariant subspace theorem for the unilateral shift 
shows that if $L \in \Lat(\Hinf)$, then there is some inner
function $\omega$ so that $L = \omega H^2$.  
In particular, every invariant subspace is cyclic.
Observe that 
\[
 N_\omega := \omega H^2 \ominus \omega B_E H^2 
 = \omega (H^2 \ominus B_E H^2) =\omega  M(E).
\]
Here $M(E) = \spn\{k_\lambda : \lambda \in E\} = H^2 \ominus B_EH^2$.
However $M_\omega$ is an isometry in $\fM(H^2)$,
and hence commutes with $\fM(H^2)$.  So it provides a unitary equivalence
between the compression of $\fM(H^2)$ to $N_\omega$ 
and to $M(E)$.
  
Therefore, for $f \in \Hinf$, Theorem~\ref{T:A1NP} gives
\begin{align*}
 \dist (f,B_E\Hinf) &= 
 \sup \{\| P_{N_\omega} M_f^* |_{N_\omega} \| : \omega \text{ inner} \} = 
 \| M_f^*|_{M(E)} \|.
\end{align*}
The classical Nevan\-linna-Pick interpolation theorem now follows
as in Corollary~\ref{C:A1NP}, as Sarason showed.
\smallbreak

Perhaps more importantly, Theorem~\ref{T:A1NP} provides an interpolation 
theorem for \emph{any} weak-$*$-closed subalgebra of $\Hinf$.  
These constrained Nevan\-linna-Pick theorems fashion suitable 
generalizations of the results seen in \cite{DPRS09} and \cite{Rag09a}. 

In fact, since any weak-$*$ linear functional $\phi$ on $\fA$ extends 
to a functional on $\Hinf$ with arbitrarily small increase in norm, 
it can be factored as $\phi = [hk^*]$ 
where $h$ is outer and $\|h\|\,\|k\| < \|\phi\| + \ep$.
Therefore the more refined Theorem~\ref{T:A1ML} applies.
We obtain: 

\begin{thm} \label{T:subHinf}
Let $\fA$ be a weak-$*$-closed unital subalgebra of $\Hinf$.
Let $E = \{ \lambda_1,\dots,\lambda_n \}$ be a finite subset 
of $\bD$ which is separated by $\fA$.
Then there is a multiplier in the unit ball of $\fA$ with $f(\lambda_i) = w_i$
for $1 \le i \le n$ if and only if
\begin{align*}
 \Big[ (1 - w_i\ol{w_j}) \ip{P_L k_{\lambda_i},k_{\lambda_j}} \Big] \geq 0
\end{align*}
for all cyclic subspaces $L= \fA[h]$ where $h$ is outer.
\end{thm} 

In \cite{DPRS09}, the algebra $\Hinf_1 = \{f\in \Hinf : f'(0)=0\}$ is studied. 
Beurling's Theorem was used to show that there is a simple 
parameterization of the cyclic invariant subspaces $\Hinf_1[h]$
for $h$ outer, namely 
\[
 H^2_{\alpha, \beta} := \spn \{ \alpha + \beta z,\ z^2H^2 \} 
 \qforal (\alpha,\beta) \text{ with }|\alpha|^2 + |\beta|^2 =1.
\]
It was shown that these provide a Nevan\-linna-Pick family for $\Hinf_1$,
and the consequent interpolation theorem.
Raghupathi \cite{Rag09a} carries out this program for the class of
algebras $\Hinf_B = \bC 1 + B \Hinf$ where $B$ is a finite Blaschke product.

In \cite{Rag09b}, Raghupathi shows that Abrahamse's
interpolation result for multiply connected domains \cite{Abram79} 
is equivalent to the interpolation problem for  certain
weak-$*$ subalgebras of $\Hinf$.  Consequently, 
our distance formula includes Abrahamse's theorem
as a special case.

%%%%%%%%%%%%%%%%%%%%
\subsection*{Singly-generated multiplier algebras}
Suppose that $T$ is an absolutely continuous contraction on $H$, 
and let $\fA_T$ be the unital, weak-$*$-closed algebra generated by $T$.  
The Sz.Nagy dilation theory provides a weak-$*$ continuous, 
contractive homomorphism 
$\Phi : \Hinf \to \fA_T$  given by $\Phi(f) = P_H f(U)|_H$, 
where $U$ is the minimal unitary dilation of $T$.  
If $\Phi$ is isometric (i.e.\ $\| \Phi(f) \| = \| f \|_\infty$ for every $f \in \Hinf$), 
then $\Phi$ is also a weak-$*$ homeomorphism.  
In addition, $\Phi$ being isometric ensures that the
preduals $(\fA_T)_*$ and $L_1 / H_0^1$ are isometrically isomorphic. 
In this case, we say that $T$ has an \emph{isometric functional calculus}.  
See \cite{BFP85} for relevant details.  
The following deep result of Bercovici \cite{Berco88} will be used.

%%%%%%%%%%%%%%%%%%%%
\begin{thm} [Bercovici] \label{T:Berc}
Suppose $T$ is an absolutely continuous contraction on $H$ 
and that $T$ has an isometric functional calculus.  
Then $\fA_T$ has property $\bA_1(1)$.
\end{thm}

We will use Theorem~\ref{T:Berc} to show that a wide class of reproducing 
kernel Hilbert spaces admit Nevan\-linna-Pick families of kernels 
for arbitrary dual subalgebras of $\Hinf$.  

%%%%%%%%%%%%%%%%%%%%
\begin{eg}
Let $\mu$ be a finite Borel measure on $\bD$, and
let $P^2(\mu)$ denote the closure of the polynomials in $L^2(\mu)$.
A \emph{bounded point evaluation} for $P^2(\mu)$ 
is a point $\lambda$ for which 
there exists a constant $M > 0$ with  
$| p(\lambda) | \leq M \| p \|_{P^2(\mu)}$ 
for every polynomial $p$. 
A point $\lambda$ is said to be an \emph{analytic bounded point 
evaluation for $P^2(\mu)$} if $\lambda$ is in the interior of the set 
of bounded point evaluatons, and that the map $z \to f(z)$
is analytic on a neighborhood of $\lambda$ for every $f \in P^2(\mu)$.

It follows that if $\lambda$ is a bounded point evaluation, 
then there is a kernel function $k_\lambda$ in $P^2(\mu)$ 
so that $p(\lambda) = \ip{ p, k_\lambda }$.  
For an arbitrary $f \in P^2(\mu)$, if we set $f(\lambda) = \ip{ f,k_\lambda}$, 
then these values will agree with $f$ a.e.\  with respect to $\mu$.  
For both the Hardy space and Bergman space, the set of analytic 
bounded point evaluation is all of $\bD$.  
A theorem of Thomson shows that either $P^2(\mu) = L^2(\mu)$ 
or $P^2(\mu)$ has analytic bounded point evaluations \cite{Tho91}.

Let $m$ be Lebesque area measure on the disk. 
For $s > 0$, define a weighted area measure on $\bD$ by 
$d\mu_s(z) = (1 - | z | )^{s-1}dm(z)$.  
The monomials $z^n$ form an orthogonal basis for $P^2(\mu)$. 
This includes the Bergman space for $s=1$.  
For these spaces, every point in $\bD$ is an analytic point evaluation.
\end{eg}

The following result appears as Theorem 4.6 in \cite{AMc02}.

%%%%%%%%%%%%%%%%%%%%
\begin{thm} \label{T:bpe}
Let $\mu$ be a measure on $\bD$ such that  the set of analytic bounded 
point evaluations of $P^2(\mu)$ contains all of $\bD$. 
Then $\fM(P^2(\mu))$ is isometrically isomorphic and weak-$*$ homeomorphic to $\Hinf$.
\end{thm}

Corollary~\ref{C:A1NP} yields the following interpolation result for these spaces.

%%%%%%%%%%%%%%%%%%%%
\begin{thm} \label{T:bpeNP}
Let $\mu$ be a measure on $\bD$ such that  the set of analytic bounded 
point evaluations of $P^2(\mu)$ contains all of $\bD$. 
Suppose $\fA$ is a dual subalgebra of $\fM(P^2(\mu))$.  
Then $\fA$ has a Nevan\-linna-Pick family of kernels.
\end{thm}

%%%%%%%%%%%%%%%%%%%%
\begin{eg} \label{E:Bergman again}
In particular, Theoorem~\ref{T:bpeNP} provides a Nevan\-linna-Pick condition for Bergman space  $A^2 := A^2(\bD)$, whose reproducing kernel 
kernel $k^B_\lambda = (1-\overline{\lambda}z)^{-2}$ is not an NP kernel.
 In fact, the Bergman kernel fails 
to even have the two-point Pick property.  See  \cite[Example~5.17]{AMc02} for details.
The multiplier algebra of Bergman space has property $\bA_1(1)$ 
as a consequence of much stronger properties (see Section~\ref{S:matrix}), 
but the subspace lattice of the Bergman shift is immense.
\end{eg}

%%%%%%%%%%%%%%%%%%%%%%%%%%%%%%%%%%%%%%%
\section{Complete NP kernels} \label{S:DA}

A reproducing kernel is a \textit{complete Nevanlinna-Pick kernel}
if matrix interpolation is determined by the positivity of the Pick
matrix for the data.  
That is, if $H$ is a Hilbert space with reproducing kernel $k$ on a set $X$,
$E=\{\lambda_1,\dots,\lambda_n\}$ is a finite subset of $X$, and 
$W_1,\dots,W_n$ are $r\times r$ matrices, 
then a necessary condition for there to be an element $F \in \M_r(\fM(H))$
with $F(\lambda_i)=W_i$ and $\|F\|\le 1$ is the positivity of the matrix
\[
 \Big[ (I_r - W_i W_j^*) k(\lambda_i,\lambda_j) \Big] .
\]
We say that $k$ is a complete NP kernel if this is also sufficient.

The first author and Pitts \cite{DP98b} showed that symmetric Fock space, 
now called the Drury-Arveson space $H^2_d$, on the complex ball $\bB_d$
of $\bC^d$ (including $d=\infty$) is a complete NP space with kernel 
\[ k (w,z) = \dfrac1{1-\ip{w,z}} .\] 

The complete NP kernels were classified by 
McCullough\cite{McCull92,McCull94} and Quiggin \cite{Quig93,Quig94}
building on work by Agler (unpublished).  Another proof was provided
by Agler and McCarthy \cite{AMc00}, who noticed the universality of the
Drury-Arveson kernel.

%%%%%%%%%%%%%%%%%%%%
\begin{thm}[McCullough, Quiggin, Agler-McCarthy] \label{T:MCQAMC}
Let $k$ be an irreducible kernel on $X$.  
Then $k$ is a complete Nevan\-linna-Pick kernel if and only if 
for some cardinal $d$, there is an injection $b: X \to \bB_d$ 
and a nowhere-vanishing function $\delta : X \to \bC$ such that
\begin{align*}
 k(x,y) =   \dfrac{\ol{\delta(x)}\delta(y)}{1-\ip{b(x),b(y)}} .
\end{align*}
In this case, the map $k_x \to \dfrac{\delta(x)}{1-\ip{z,b(x)}}$ 
extends to an isometry of $H$ into $H^2_d$.
\end{thm}

Since the span of kernel functions is always a co-invariant subspace for
the space of multipliers, the complete NP kernel spaces are seen to 
correspond to certain co-invariant subspaces of Drury-Arveson space,
i.e.\  $\spn\{ k_z : z \in b(X) \}$.
The first author and Pitts \cite{DP98b} show that the multiplier algebras of these spaces are all
complete quotients of the non-commutative analytic Toeplitz algebra, $\fL_d$,
generated by the left regular representation of the free semigroup $\Fd$
on the full Fock space $\Fock$.  It follows immediately that they
are complete quotients of $\fM(H^2_d)$.  
See also Arias and Popescu \cite{AP00}.

We will show that all such quotients of $\fM(H^2_d)$ have property $\bA_1(1)$.
The algebra $\fL_d$ actually has property $\bA_{\aleph_0}$ \cite{DP99}
and even property $X_{0,1}$ \cite{Berco98}.
But these stronger properties do not extend to $\fM(H^2_d)$.

More specifically, if $\fJ$ is a \wot-closed ideal of $\fL_d$ with 
range $M=\ol{\fJ \Fock}$, then \cite{DP98b} shows that 
$\fL_d/\fJ$ is completely isometrically isomorphic to
the compression of $\fL_d$ to $M^\perp$.
Conversely, if $M$ is an invariant subspace of both $\fL_d$
and its commutant, the right regular representation algebras $\fR_d$, 
then $\fJ = \{ A \in \fL_d : \ran A \subset M\}$ is a \wot-closed ideal
with range $M$.  
In particular, if $\fC$ is the commutator ideal, it is shown that
$\fM(H^2_d) \simeq \fL_d/\fC$.
Moreover the compression of both $\fL_d$ and $\fR_d$ to $H^2_d$
agree with $\fM(H^2_d)$.  So if $N$ is a coinvariant subspace of
$H^2_d$, then $M=\Fock \ominus N$ is invariant for both $\fL_d$
and $\fR_d$ and the theory applies.

The following result is due to Arias and Popescu \cite[Prop.1.2]{AP00}.
We provide a proof for completeness.

%%%%%%%%%%%%%%%%%%%%
\begin{thm}[Arias--Popescu] \label{T:quotients have A1}
Let $\fJ$ be any \wot-closed ideal of $\fL_d$ and let $M=\ol{\fJ \Fock}$.
Then $\fA = P_M^\perp \fL_d|_{M^\perp}$ has property $\bA_1(1)$.
\end{thm}

\begin{proof}
Let $q: \fL_d \to \fL_d / \fJ \simeq \fA$ denote the canonical quotient map.
Suppose $\phi$ is a weak-$*$ functional on $\fA$ with $\|\phi\| < 1$.  
Then $\phi \circ q$ is a weak-$*$ continuous functional on $(\fL_d)_*$
of norm less than $1$.
Hence there are vectors $x$ and $y$ in $\Fock$ with $\phi \circ q = [xy^*]$
and $\| x \| \| y \| < 1$.  
Form the cyclic subspace $L = \fL_d[x]$.
By \cite{AP95, DP99} there is an isometry $R \in \fR_d$ 
such that $\fL_d[x] = R\Fock$.  
Let $u$ be the vector such that $x = Ru$ and set $v = R^*y$.  
A direct calculation shows that $[xy^*] = [u v^*]$.
Observe that $\fL_d[u]=\Fock$.

Obviously $[uv^*]$ annihilates $\fJ$, and it follows that $v$ is orthogonal to
\begin{align*}
 \ol{\fJ u} &= \ol{\fJ \fL_d u}= \ol{\fJ \Fock} = M.
\end{align*}
Thus $v \in M^\perp$.
Now, for $A \in \fA$, pick $B \in \fL_n$ with $A = q(B) = P_M^\perp B|_M^\perp$.  
Then since $M^\perp$ is coinvariant for $\fL_d$,
\begin{align*}
\phi(A) &= \phi \circ q(B) = \ip{ Bu,v } \\&
= \ip{ Bu, P_M^\perp v } 
= \ip{  P_M^\perp B  P_M^\perp u,v } \\& 
= \ip{ A (P_M^\perp u), v }.
\end{align*}
Hence $\phi = [(P_M^\perp u) v^*]$. Therefore $\fA$ has property $\bA_1(1)$.
\end{proof}

The remarks preceding this theorem show that the multiplier algebra
of every complete NP kernel arise in this way.  
We further refine this to observe that the vector $x$ in the factorization
may be chosen so that it does not vanish.  

%%%%%%%%%%%%%%%%%%%%
\begin{cor} \label{C:DA has A1}
The multiplier algebra $\M(H)$ of every complete NP kernel has property $\bA_1(1)$.
In particular, $\fM(H^2_d)$ has property $\bA_1(1)$.
Moreover, each $\phi \in \fM(H)_*$  with $\|\phi\|<1$ 
can be represented as $\phi=[xy^*]$ such that $x$ does 
not vanish on $X$ and $\|x\|\,\|y\| < 1$.
\end{cor}

\begin{proof}
We may assume that $\M(H)$ is the compression of $\fL_d$ to the subspace
$M = \spn \{ k_\lambda : \lambda \in X \subset \bB_n\}$.
It remains only to verify that in the proof of Theorem~\ref{T:quotients have A1},
the function $P_M^\perp u$ does not vanish.  
Since $u$ is a cyclic vector for $\fL_d$, it is not orthogonal to any $k_\lambda$
because $\bC k_\lambda$ is coinvariant.
Therefore for any $\lambda\in X$, 
\[
 0 \ne \ip{u, k_\lambda} = \ip{u, P_M^\perp k_\lambda} 
 = \ip{P_M^\perp u, k_\lambda} = \delta(\lambda)^{-1} (P_M^\perp u)(\lambda) , 
\]
where $\delta$ is the scaling function of the embedding in Theorem~\ref{T:MCQAMC}.
\end{proof}

The proof actually shows that it suffices to use vectors $h$ which are cyclic for $\fM(H)$.
In the case of Drury-Arveson space, like for Hardy space, these vectors are
called outer.

%%%%%%%%%%%%%%%%%%%%
\begin{rem}
The algebra $\fM(H^2_d)$ does not have property $\bA_{\aleph_0}$.
The reason is that algebras with this property have non-trivial invariant
subspaces which are orthogonal \cite{Berco87}.
But any two non-trivial invariant subspaces of $\fM(H^2_d)$ have non-trivial
intersection.  Indeed, if $M \in \Lat \fM(H^2_d)$, then $N=M+(H^2_d)^\perp$
is invariant for $\fL_d$.  By  \cite{AP95, DP99}, this space is the direct sum
of cyclic invariant subspaces $N_i = R_i\Fock$ for isometries $R_i \in \fR_d$.
The compression $P_{H^2_d} R_i|_{H^2_d}$ is a multiplier $M_{f_i}$
in $\fM(H^2_d)$.  Thus 
\begin{align*}
 M &= P_{H^2_d} N = \sum_i P_{H^2_d} R_i \Fock \\&
 = \sum_i P_{H^2_d} R_i P_{H^2_d} \Fock
 = \sum_i M_{f_i} H^2_d .
\end{align*}
In particular, every invariant subspace $M$ contains the range 
of a non-zero multiplier $M_f$.
Hence given two invariant subspaces $M$ and $N$ in $H^2_d$, 
we can find non-zero multipliers $f$ and $g$ with $\ran M_f \subset M$
and $\ran M_g \subset N$.  So $M \cap N$ contains $\ran M_{fg}$.

Based on heuristic calculations, we expect that $\fM(H^2_d)$ 
does not have property $\bA_2(r)$ for any $r\ge1$, 
and likely not $\bA_2$.
\end{rem}

We are now in a position to apply Theorem~\ref{T:A1ML}.

%%%%%%%%%%%%%%%%%%%%
\begin{thm}
Suppose $k$ is an irreducible, complete Nevan\-linna-Pick kernel on $X$.  
Then any dual subalgebra $\fA$ of multipliers of $H$ admits a Nevan\-linna-Pick 
family of kernels.   More specifically, if $E= \{\lambda_1,...,\lambda_n\}$
is a finite subset of $X$ which is separated by $\fA$
and $w_1,\dots,w_n$ are scalars, then  
there is a multiplier $f$ in the unit ball of $\fA$ with $f(\lambda_i) = w_i$
for $1 \le i \le n$ if and only if
\begin{align*}
 \Big[ (1 - w_i\ol{w_j}) \ip{P_L k_{\lambda_j},k_{\lambda_i}} \Big] \geq 0
\end{align*}
for all cyclic  invariant subspaces $L = \fA[h]$ of $\fA$
$($and it suffices to use $h$ which do not vanish on $E)$. 
\end{thm}

%%%%%%%%%%%%%%%%%%%%%%%%%%%%%%%%%%%%%%
\section{Finite Kernels} \label{S:finite}

In this section, we present numerical evidence that there is a finite 
dimensional multiplier algebra $\fA$ with the property that 
$\CycLat(\fA)$ is not an NP family for $\fA$.  In particular, this algebra does not have $\bA_1(1)$ . It does, however,  have property  $\bA_1(r)$ for some $r > 1$.

Suppose $X = \{\lambda_1,\dots,\lambda_N\}$ is a finite set and 
$k : X \times X \rightarrow \bC$ is an irreducible kernel.  
Let $y_1,\dots,y_N$ be vectors in $\bC^N$ such that
$k(\lambda_i,\lambda_j)  = \ip{ y_j,y_i}$, and 
let $\{x_1,\dots,x_N\}$ be a dual basis for the $y_i$.
The space $H = \bC^N$ may  be regarded as a reproducing kernel Hilbert space
over $X$, with reproducing kernel at $\lambda_i$ given by $y_i$.
The multiplier algebra $\fM(H)$ is an $N$-idempotent operator algebra spanned by
the rank one idempotents $p_i = x_i y_i^*$.

If $\{e_i\}$ is the canonical orthonormal basis for $\bC^N$, 
then one readily sees that $\fM(H)$ is similar to the diagonal 
algebra $\fD_N$ via the similarity $ S$ defined by $Se_i = x_i$.  
Since $\fD_N$ evidently has property $\bA_1(1)$, it follows from 
elementary results on dual algebras that
$\fM(H)$ has $\bA_1(r)$ for some $r \geq 1$.

If $k$ is irreducible and a complete NP kernel, then Corollary~\ref{C:DA has A1} 
shows that $\fM(H)$ has $\bA_1(1)$.  
However, there are many kernels $k$ that cannot
be embedded in Drury-Arveson space in this way.  
We expect that many of these algebras fail to have $\bA_1(1)$ 
and that the distance formula fails in such cases.  

Since $\fA$ is similar to the diagonal algebra $\fD_N$, the invariant
subspaces are spanned by some subset of $\{x_1,\dots, x_N\}$.
Denote them by $L_\sigma = \spn\{x_i : i \in \sigma\}$.
For $E \subset \{1,\dots,N\}$, the ideal $\fJ = \fJ_E = \spn\{ p_i : i \not\in E \}$.
Then $\fJ L_\sigma = L_{\sigma\setminus E}$, and
$N_\sigma := N_{L_\sigma} = L_\sigma \ominus L_{\sigma \setminus E}$.
The distance formula is obtained as the maximum of compressions to 
these subspaces---so we need only consider the maximal ones.
These arise from $\sigma \supset E$.

For trivial reasons, the distance formula is always satisfied when $N = 2$ and $N=3$.  
There is strong numerical evidence to suggest that the formula does hold for $N=4$, 
though we do not have a proof.  In the following $5$-dimensional example, 
\emph{Wolfram Mathematica 7} was used to find a similarity $S$ such
that the distance formula fails.  

\begin{eg}
Define the similarity
\[
S= \left[ \begin{array}{rrrrr}   
3&1 & 1 & 0 & -1  \\ 
0 & 1 & -2 & -1 &0 \\ 
-1 & 0 & -1 & 1 &-1 \\ 
-1 & 1 & 2 & 1 & -1  \\ 
1 & 1 & 3 &1 & -2      
\end{array} \right]
\]
Let $p_i = x_i y_i^*$ for $1 \le i \le 5$ be the idempotents
which span the algebra $\fA := \fM(H)$. 
Let $E = \{1,2,3\}$, and form $\fJ =\fJ_E = \spn\{ p_4, p_5 \}$.  

Consider the element $A = -2p_1 -3p_2 + 7p_3$.
We are interested in comparing 
$\max_{\sigma}  \| P_{N_\sigma} A P_{N_\sigma} \| $
with $\dist(A,\fJ)$.
As noted above, it suffices to use maximal $N_\sigma$'s formed by
the cyclic subspaces that do not vanish on $E$, namely
\begin{align*}
 N_{\{123\}} &= \spn \{x_1,x_2,x_3\},\\ 
 N_{\{1234\}} &= \spn\{x_1,x_2,x_3,x_4\} \ominus \bC x_4,\\ 
 N_{\{1235\}} &= \spn\{x_1,x_2,x_3,x_5\} \ominus \bC x_5,\AND \\
 N_{\{12345\}} &= \spn\{x_4,x_5\}^\perp = \spn\{y_1,y_2,y_3\}. 
\end{align*}
For notational convenience, set $P_\sigma := P_{N_\sigma}$. 
The values of 
$ \| P_\sigma A P_\sigma \| $
were computed and rounded to four decimal places:
\begin{align*}
\| P_{123} A P_{123}\| &= \phantom{0}9.0096 ,\\
\| P_{1234} A P_{1234}\| &=10.1306 , \\
\| P_{1235} A P_{1235}\| &= \phantom{0}7.4595 ,\\
\| P_{12345} A P_{12345}\| &= 10.6632 .
\end{align*}
By minimizing a function of two variables, the following distance estimate was obtained
\begin{align*}
\dist(A,\fJ) \approx 11.9346.
\end{align*}
Similar results appeared for many different elements of $\fA$, which indicate that $\CycLat(\fA)$ is not an NP family for $\fA$.
Consequently, it must also fail to have $\bA_1(1)$.  We currently have no example of a dual algebra of multipliers  on any $H$ that fails to have $\bA_1(r)$ for every $r \geq 1$, or even fails to have $\bA_1$.  
\end{eg}

%%%%%%%%%%%%%%%%%%%%%%%%%%%%%%%%%%%%%%
\section{Matrix-valued Interpolation} \label{S:matrix}

In this section, we will discuss matrix-valued Nevan\-linna-Pick interpolation problems.  
The classical theorem for matrices says that given $z_1,...,z_n$ in the disk, 
and $r \times r$ matrices $W_1,...,W_n$, there is a function $F$ in the the unit
ball of $\M_r(\Hinf)$ such that $F(z_i) = W_i$ if and only if the Pick matrix
\begin{align*}
 \Bigg[ \frac{I_r-W_iW_j^*}{1-z_i\ol{z_j}} \Bigg]_{r\times r}
\end{align*}
is positive semidefinite.  

One can define a linear map $R$ on $M(E) \otimes \bC^r$ by setting 
\[
 R (k^s_{\lambda_i} \otimes u) = k^s_{\lambda_i} \otimes W_i^*u 
 \qfor 1 \le i \le n \AND u \in \bC^r .
\]
Note that if $F$ is an arbitrary interpolant, then $R = M_F^*|_{M(E) \otimes \bC^r}$.
Now $\|R\|\le 1$ is equivalent to $I-R^*R \ge 0$, which is equivalent
to the positivity condition above.  
Hence, this provides the complete distance formula:
\begin{align*}
 \dist(F, \M_r(\fJ)) = \| M_F^*|_{M(E) \otimes \bC^r}\| .
\end{align*}
The same holds (by definition) for all complete NP kernels when
the factor $\frac1{1-z_i\ol{z_j}}$ is replaced by $k(\lambda_i,\lambda_j)$.
The multiplier algebra of all complete NP kernels therefore satisfy the analogous distance formula.

Our goal is to generalize the results of Section~\ref{S:NPfamilies} 
to a matrix-valued setting by imposing stronger conditions on our 
algebras of multipliers.
Let $(H,k)$ be a reproducing kernel Hilbert space over $X$, and let $r \geq 1$. 
We will consider the algebra $\M_r(\fM(H))$ of $r\times r$ matrices 
of multipliers acting on the vector-valued  space $H^{(r)} = H \otimes \bC^r$.

For any non-zero vector $u \in \bC^r$, the functions 
$k_\lambda \otimes u$ act as vector-valued kernel functions. 
We can therefore consider  $\M_r(\fM(H))$ as functions on $X$ with values
in $\M_r$.  They act as multipliers of $H^{(r)}$, and inherit a norm
as elements of $\M_r(\B(H)) \simeq \B(H^{(r)})$.

It is readily verified that for any multiplier $F$ 
and any $\lambda \in X$, we have 
\[
 M_F^* (k_\lambda \otimes u) = k_\lambda \otimes F(\lambda)^*u
 \qfor \lambda \in X \AND u \in \bC^r. 
\]
Conversely, any bounded operator on $H^{(r)}$ that satisfies 
these relations is a multiplier of $H^{(r)}$.  
The algebra $\M_r(\fM(H))$ is a unital, weak-$*$-closed subalgebra 
of $\M_r(\B(H))$, and thus is a dual algebra.  
Consequently, we may apply the same heuristic as 
Section~\ref{S:NPfamilies} when trying to compute distances.

Any dual subalgebra $\fA$ of $\fM(H)$ determines the
dual subalgebra $\M_r(\fA)$ of   $\M_r(\fM(H))$.
Suppose that $E=\{\lambda_i : 1 \le i \le n \}$ is a 
finite subset of $X$ separated by $\fA$.
Let $\fJ_E$ be the ideal of functions in $\fA$ vanishing on $E$.  
For $F \in \M_r(\fA)$, any subspace of the form $L^{(r)}$ for 
$L \in \Lat(\fA)$ is invariant for $\M_r(\fA)$.
Conversely, any invariant subspace of $\M_r(\fA)$ takes this form.  

The subspace $L^{(r)}$ is cyclic if and only if $L$ is $r$-cyclic
because if $x_1,\dots,x_r$ is a cyclic set, then $x=(x_1,\dots,x_r)$
is a cyclic vector for $L^{(r)}$ and vice versa.
So in general we cannot deal only with cyclic invariant subspaces
of the algebra $\fA$.  We will have to deal with some multiplicity
of the kernels on these spaces.
This can be handled as in the discussion in Remark~\ref{R:multiplicity}.

We can apply Lemma~\ref{L:dist1} to $\M_r(\fA)$ and the ideal
$\M_r(\fJ_E)$.  For any $F \in \M_r(\fA)$, we have
\begin{align*}
 \dist(F, \M_r(\fJ_E)) &\ge 
 \sup_{L \in \Lat(\fA)} \| (P_{N_L} \otimes I_r) M_F (P_{N_L} \otimes I_r)  \| \\&
 \ge  \sup_{L \in \Lat(\fA)} \| (P_{M_L} \otimes I_r) M_f (P_{M_L} \otimes I_r) \|.
\end{align*}

%%%%%%%%%%%%%%%%%%%%
\begin{defn}
If equality holds for this distance formula for every finite subset $E$
which is separated by $\fA$, i.e.\ for any $F \in \M_r(\fA)$, and
$N_L = L \ominus \fJ_E L$ for $L \in \Lat \fA$, 
\[
 \dist(F, \M_r(\fJ_E)) =
 \sup_{L \in \Lat(\fA)} \| (P_{N_L} \otimes I_r) M_F (P_{N_L} \otimes I_r)  \|
\]
then we say that $\Lat \fA$ is an \textit{$r\times r$ Nevan\-linna-Pick family} 
for $\fA$.  
If this holds for all $r \geq 1$, then we say that $\Lat \fA$
is a \emph{complete Nevan\-linna-Pick family} for $\fA$.
\end{defn}

Generally, property $\bA_1(1)$ is not inherited by matrix algebras.
Conway and Ptak \cite{CP02} show that any absolutely
continuous contraction in class $C_{00}$ with an isometric functional
calculus has $X_{0,1}$.  This includes the Bergman shift $B$, and consequently the multiplier algebra of Bergman space.
The property $X_{0,1}$ implies that $\M_r(\fA)$ has $\bA_1(1)$ for every $r\ge1$.  

On the other hand, it can be the case that some finite ampliation
of the algebra will have $\bA_1(1)$.
Given a dual algebra $\fA$ on a Hilbert space $H$, 
the $k$-th ampliation $\fA^{(k)}$ is an isometric representation
of $\fA$ on $H^{(k)}$, the direct sum of $k$ copies of $H$,
with elements $A^{(k)} = A \oplus \dots \oplus A$, the direct sum of $k$ copies of $A$.
The preduals $\fA_*$ and $\fA^{(k)}_*$ are isometrically isomorphic.
However, a rank $k$ functional on $\fA$ converts to a rank one functional
on $\fA^{(k)}$ since
\[
 \sum_{i=1}^k \ip{A x_i, y_i} =  \ip{T^{(k)} \mathbf{x},\mathbf{y}} 
\]
where $\mathbf{x}=(x_1,\dots, x_k)$ and 
$\mathbf{y} = (y_1,\dots, y_k)$ in $H^{(k)}$.
If $\fA$ has $\bA_1$, so does $\M_r(\fA^{(r^2)})$ (Proposition 2.6 \cite{BFP85}),
but the constants are not always good enough.

The infinite ampliation of any operator algebra has $\bA_1(1)$.
This is because weak-$*$ continuous functional on $\B(H)$  can be
represented by a trace class operator.  Using the polar decomposition,
this can be realized as $\phi = \sum_{i=1}^\infty [x_i y_i^*]$ where
$\sum_i \|x_i\|^2 = \sum_i \|y_i\|^2 = \|\phi\|$.  So
\[
 \phi(T) = \sum_{i=1}^\infty \ip{Tx_i,y_i} = 
 \ip{T^{(\infty)} \mathbf{x},\mathbf{y}} 
\]
where $\mathbf{x}=(x_1,x_2,\dots)$ and 
$\mathbf{y} = (y_1,y_2,\dots)$ in $H^{(\infty)}$.
In fact, this infinite ampliation is easily seen to have property $X_{0,1}$.

As in Theorem~\ref{T:A1NP}, if $\M_r(\fA)$ has property $\bA_1(1)$,
then we obtain an exact distance formula which yields a
Nevan\-linna-Pick type theorem for these algebras.
The proof is the same.

\begin{thm} \label{T:matrixNP}
Suppose $\fA$ is a dual algebra of multipliers on $H$. 
If $\M_r(\fA)$ has property $\bA_1(1)$, then 
$\Lat \fA$ is an $r\times r$ Nevan\-linna-Pick family for $\fA$. 

More generally, if the ampliation $\M_r(\fA^{(s)})$ has 
$\bA_1(1)$, then $\Lat(\fA^{(s)})$ is an $r\times r$ 
Nevan\-linna-Pick family for $\fA$. In particular, $\Lat (\fA^{(\infty)})$ is a complete Nevan\-linna-Pick family 
for any algebra of multipliers $\fA$.
\end{thm}

While it appears that ampliations of matrix algebras over some well known 
multiplier algebras have $\bA_1(1)$, we are unaware of any general results
of this kind.  Such a result would be interesting.

We will illustrate Theorem~\ref{T:matrixNP} with some examples.

\subsection*{Bergman Space} 
The Bergman shift $B$ on $A^2(\bD)$ has property $X_{0,1}$ \cite{CP02}.
This is inherited by any dual subalgebra $\fA$ of $\fM(A^2(\bD))$.
Therefore $\M_r(\fA)$ has property $\bA_1(1)$ for all $r\ge1$.
We obtain a formulation of the complete Nevan\-linna-Pick
interpolation for subalgebras of $\Hinf$ in this context.

\begin{thm} \label{T:BergmanCNP}
Let $\fA$ be a dual subalgebra of $\fM(A^2(\bD)) \simeq \Hinf$.
Let $E = \{z_1,\dots,z_n \}$ be points in $\bD$ which are separated by $\fA$,
and let $W_1,\dots,W_n$ be $r\times r$ matrices.
There is an element $F \in \M_r(\fA)$ with $F(z_i)=W_i$
and $\|F\| \le 1$ if and only if the following holds:
for each $L\in \Lat \fA$, (setting 
%$N_L = L \ominus \fJ_E L$ and 
$N_{z_i} = L \ominus \fJ_{z_i} L$ for $1 \le i \le n$), we have
\[
  \Big[ 
  (I_r- W_iW_j^*) \otimes P_{N_{z_i}} P_{N_{z_j}} 
  \Big]_{n\times n} \ge 0
  \qforal L \in \Lat \fA.
\]
\end{thm}

\begin{proof}
Let $A$ be any element of $\M_r(\fA)$ such that $A(z_i) = W_i$.
Since $\M_r(\fA)$ has $\bA_1(1)$, Theorem~\ref{T:dist_ideal} implies that
\[
 \dist(A, \M_r(\fJ_E)) = 
 \sup_{L \in \Lat(\fA)} \| (P_{N_L} \otimes I_r) A (P_{N_L} \otimes I_r) \|.
\]
As before, a necessary and sufficient condition for interpolation with an element
$F\in\M_r(\fA)$ of norm at most one is that $\dist(A, \M_r(\fJ_E)) \le 1$.
Also arguing in a standard manner, using the semi-invariance of $N_L$,
\begin{align*}
 \| (P_{N_L} \otimes I_r) A (P_{N_L} \otimes I_r) \|^2 =
 \| (P_{N_L} \otimes I_r) A A^* (P_{N_L} \otimes I_r) \| .
\end{align*}
This has norm at most 1 if and only if
\[ (P_{N_L} \otimes I_r) ( I - A A^*) (P_{N_L} \otimes I_r) \ge 0 .\]

As we observed in Remark~\ref{R:multiplicity}, $N_L$ is spanned by the
spaces $N_{z_i}$ for $1 \le i \le n$.  
These subspaces are eigenspaces for $(P_L \fA|_L)^*$, and thus 
they are independent, and at a positive angle to each other.
So positivity of the operator above is equivalent to the positivity of
\[
 \Big[ 
  (P_{N_{z_i}} \otimes I_r) ( I -A A^*) (P_{N_{z_j}} \otimes I_r)
 \Big]= 
 \Big[ 
  (I_r- W_iW_j^*) \otimes P_{N_{\lambda_i}} P_{N_{\lambda_j}} 
 \Big]
\]
because the restriction of $(P_L\otimes I_r)A^*$ to $N_{z_j}\otimes \bC^r$ is
just $P_{N_{z_j}} \otimes W_j^*$.
\end{proof}

\subsection*{Hardy Space}
%%%%%%%%%%%%%%%%%%%%
We return to the case of subalgebras of $\Hinf$ acting on Hardy space. 
In \cite{DPRS09}, for $\fA = \Hinf_1 := \{ f \in \Hinf : f'(0) = 0 \}$,
it was shown that the distance formula for matrix interpolation \emph{fails} for $\fA$.
In our terminology, $\Lat \Hinf_1$ is not a complete Nevan\-linna-Pick family.

So we cannot drop the assumption that $\M_r(\fA)$ has $\bA_1(1)$.  
Indeed, the unilateral shift fails to have even property $\bA_2$ (Theorem 3.7 \cite{BFP83}).
We will show that with ampliations, a general result can be obtained.
The following result should be well known, but we do not have a reference.  A version of it appears as Theorem 4 in \cite{Sar67}.

\begin{lem} \label{L:A1MnHinf}
$\M_r(\Hinf)$ acting on $(H^2 \otimes \bC^r)^{(r)}$ as 
$r\times r$ matrices over $\fM(H^2)$
ampliated $r$ times has property $\bA_1(1)$.
\end{lem}

\begin{proof}
Form the infinite ampliation $\M_r(\fM(H^2)^{(\infty)})$.
Then any weak-$*$ continuous functional $\phi$ on $\M_r(\Hinf)$
with $\|\phi\| < 1$ can be represented as a rank one functional $[xy^*]$ 
on $(H^2 \otimes \bC^r)^{(\infty)} \simeq  H^{2(\infty)} \otimes \bC^r$
with $\|x\|\,\|y\| < 1$.  Write $x = (x_1,\dots,x_r)$
and $y = (y_1,\dots,y_r)$ with $x_i$ and $y_i$ in $H^{2(\infty)}$
so that if $F = \big[ f_{ij} \big] \in \M_r(\Hinf)$, then
\[ \phi(F) = \sum_{i,j=1}^r \ip{ M_{f_{ij}} x_j, y_i} .\]

Let $M = \Hinf[x_1,\dots,x_r]$.
By the Beurling-Lax-Halmos theory for shifts of infinite multiplicity \cite{Halmos61},
$M_z^{(\infty)}|_M$ is unitarily equivalent to $M_z^{(s)}$ for some $s\le r$.
Thus we may assume that $x_i$ and $y_j$ live in $H^{2(r)}$.
So this means that $x$ and $y$ are then identified with vectors in
$(H^2 \otimes \bC^r)^{(r)}$ as desired.
\end{proof}

We immediately obtain:
 
%%%%%%%%%%%%%%%%%%%%
\begin{thm} \label{T:subHinfinity}
Suppose $\fA$ is a dual subalgebra of $\Hinf$ acting on $H^2$.  
Then $\Lat(\fA^{(r)})$ is an $r \times r$ Nevan\-linna-Pick family for $\fA$.
\end{thm}

An additional application of the Beurling-Lax-Halmos result shows 
that Theorem~\ref{T:subHinfinity} reduces to the matrix-valued Nevan\-linna-Pick theorem when $\fA = \Hinf$.

 In the case of $\Hinf_1$, this yields the result of Ball, Bolotnikov and Ter Horst
\cite{BBtH}.  They express their models as invariant subspaces of $\M_r(H^2)$
(in the Hilbert Schmidt norm) 
instead of $H^{2(r)} \otimes \bC^r$, but this is evidently the same space.
It suffices to use subspaces which are cyclic for $\Hinf$.
In much the same manner as \cite{DPRS09}, they obtain an explicit
parameterization of these subspaces.

%%%%%%%%%%%%%%%%%%%%%%%%%%%%%%%%%%
 
%%%%%%%%%%%%%%%%%%%%%%%%%%%%%%


\begin{thebibliography}{99}

\bibitem{Abram79} 
M.B. Abrahamse, 
\emph{The Pick interpolation theorem for finitely connected domains}, 
Michigan Math.\ J. \textbf{26} (1979), 195--203.

\bibitem{AMc00} 
J. Agler and J.E. McCarthy, 
\emph{Complete Nevan\-linna-Pick kernels}, 
J. Funct.\ Anal. \textbf{175} (2000), 111--124.

\bibitem{AMc02} 
J. Agler and J.E. McCarthy, 
\emph{Pick interpolation and Hilbert function spaces}, 
Graduate Studies in Mathematics  \textbf{44}, 
Amer.\ Math.\ Soc., Providence RI, 2002.

\bibitem{ABFP} C. Apostol, H. Bercovici, C. Foia\c s, and C. Pearcy, 
\textit{Invariant subspaces, dilation theory, and the structure 
of the predual of a dual algebra. I}, 
J. Funct.\ Anal. \textbf{63} (1985), 369--404. 

\bibitem{AP95}
A. Arias and G. Popescu, 
\textit{Factorization and reflexivity on Fock spaces},
Integral Equations Operator Theory \textbf{23} (1995),268--286.

\bibitem{AP00}
A. Arias and G. Popescu, 
\textit{Noncommutative interpolation and Poisson transforms},
Israel J. Math. \textbf{115} (2000), 205--234.

\bibitem{BBtH}
J. Ball, V. Bolotnikov, and S. ter Horst, 
\textit{A constrained Nevan\-linna-Pick interpolation problem for matrix-valued functions}, 
Indiana Univ.\ Math.\ J., to appear. arXiv:0809.2345v1

\bibitem{BFP83} 
H. Bercovici, C. Foias and C. Pearcy, 
\emph{Dilation theory and systems of simultaneous equations in the predual of an operator algebra}, 
Michigan Math.\ J. \textbf{30} (1983), 335-354.

\bibitem{BFP85} 
H. Bercovici, C. Foias and C. Pearcy, 
\emph{Dual algebras with applications to invariant subspaces and dilation theory}, 
CBMS Notes \textbf{56}, 
Amer.\ Math.\ Soc., Providence RI, 1985.

\bibitem{Berco87} 
H. Bercovici, 
\emph{A note on disjoint invariant subspaces}, 
Michigan Math.\ J. \textbf{34} (1987), 435--439.

\bibitem{Berco88} 
H. Bercovici, 
\emph{Factorization theorems and the structure of operators on Hilbert space}, 
Ann.\ of Math.\ (2) \textbf{128} (1988), 399--413.

\bibitem{Berco98}
H. Bercovici, 
\emph{Hyper-reflexivity and the factorization of linear functionals}, 
J. Funct.\ Anal. \textbf{158} (1998), 242--252.

\bibitem{CP02} 
Conway, J.B. and Ptak, M., 
\emph{The harmonic functional calculus and hyperreflexivity}, 
Pacific J. Math. \textbf{204} No. 1 (2002).

\bibitem{DPRS09} 
K.R. Davidson, V. Paulsen, M. Raghupathi and D. Singh, 
\emph{A constrained Nevan\-linna-Pick interpolation problem}, 
Indiana Univ.\ Math.\ J., \textbf{58} (2009), 709--732.

\bibitem{DP99} 
K.R. Davidson and D.R. Pitts, 
\emph{Invariant subspaces and hyper-reflexivity for free semigroup algebras}, 
Proc.\ London Math\ Soc \textbf{78} (1999), 401--430.

\bibitem{DP98a} 
K.R. Davidson and D.R. Pitts, 
\emph{The Algebraic structure of non-commutative analytic Toeplitz algebras}, 
Math. Ann. \textbf{311} (1998), 275--303.

\bibitem{DP98b}
K.R. Davidson and D.R. Pitts, 
\emph{Nevan\-linna-Pick interpolation for non-commutative analytic Toeplitz algebras}, Integral Eqtns.\ and Operator Theory \textbf{31} (1998), 321--337.

\bibitem{FV01} 
S.I. Federov, V.L. Vinnikov, 
\emph{On the Nevan\-linna-Pick interpolation in multiply connected domains}, 
translation in J. Math.\ Sci.\ (New York), \textbf{105} (2001), 2109--2126.

\bibitem{HN82}
D.W. Hadwin and E.A. Nordgren,
\textit{Subalgebras of reflexive algebras},
J. Operator Theory \textbf{7} (1982), 3--23.

\bibitem{Halmos61}
P. Halmos, 
\textit{Shifts on Hilbert spaces},
J. Reine Angew.\ Math. \textbf{208} (1961), 102--112.

\bibitem{Kenn1} 
M. Kennedy,
\textit{Wandering vectors and the reflexivity of free semigroup algebras},
J. Reine Angew.\ Math., to appear (arXiv:0909.3479). 

\bibitem{Kenn2} 
M. Kennedy,
\textit{Absolutely continuous representations of the non-commutative disk algebra},
preprint arXiv:1001.3182 

\bibitem{KL86}
J. Kraus and D. Larson,
\textit{Reflexivity and distance formulae},
Proc.\ London Math.\ Soc. \textbf{53} (1986), 340--356.

\bibitem{McCull92}
S. McCullough, 
\textit{CarathŽodory interpolation kernels}, 
Integral Equations Operator Theory \textbf{15} (1992), 43--71. 

\bibitem{McCull94}
S. McCullough, 
\textit{The local de Branges-Rovnyak construction and complete Nevan\-linna-Pick kernels}, 
in  ``Algebraic Methods in Operator Theory'', pp. 1524, Birkhauser, Basel, 1994. 

\bibitem{McP02} 
S. McCullough and V. Paulsen, 
\emph{$C^*$-envelopes and interpolation theory}, 
Indiana Univ.\ Math.\ J., \textbf{51} (2002), no. 2, 479--505.

\bibitem{Pau01} 
V. Paulsen, 
\emph{Operator Algebras of Idempotents}, 
J. Funct.\ Anal.\ \textbf{181} (2001), 209--226.

\bibitem{Pau-Notes} 
V. Paulsen, 
\emph{An introduction to the theory of reproducing kernel Hilbert spaces}, 
course notes, available at http://www.math.uh.edu/vern.

\bibitem{Quig93}
P. Quiggin, 
\textit{For which reproducing kernel Hilbert spaces is Pick's theorem true?},
Integral Equations Operator Theory \textbf{16} (1993), 244--266.

\bibitem{Quig94}
P. Quiggin, 
\textit{Generalisations of Pick's Theorem to Reproducing Kernel Hilbert Spaces},
Ph.D. thesis, Lancaster University, 1994. 

\bibitem{Rag09a} 
M. Raghupathi, 
\emph{Nevan\-linna-Pick interpolation for $\bC+B\Hinf$}, 
Integral Eqtns.\ and Operator Theory  \textbf{63} (1) (2009), 103--125.

\bibitem{Rag09b}
M. Raghupathi, 
\emph{Abrahamse's interpolation theorem and Fuchsian groups}, 
J. Math.\ Anal.\ Appl. \textbf{355} (1) (2009), 258--276.

\bibitem{Sar65} D. Sarason, 
\emph{The $H^p$ spaces of an annulus},
Mem.\ Amer.\ Math.\ Soc.\ No. \textbf{56} (1965).

\bibitem{Sar67} D. Sarason, 
\emph{Generalized interpolation in $\Hinf$}, 
Trans.\ Amer.\ Math.\ Soc. \textbf{127} (1967), 179--203.

\bibitem{Tho91} 
J. Thomson, 
\emph{Approximation in the mean by polynomials}, 
Annals of Math. \textbf{133} (1991), 477--507.

\end{thebibliography}
\end{document}